\documentclass[12pt]{amsart}
\usepackage{amsmath,amsfonts,amsthm,amscd,amssymb,mathrsfs,amssymb, multirow}
\usepackage{stmaryrd}
\usepackage{graphicx}
\newtheorem{theorem}{Theorem}

\newtheorem{proposition}[theorem]{Proposition}
\newtheorem{lemma}[theorem]{Lemma}
\newtheorem{definition}[theorem]{Definition}
\newtheorem{conjecture}[theorem]{Conjecture}

\numberwithin{figure}{section}
\newtheorem{remark}[theorem]{Remark}

\newcommand{\CP}{\mathbb{CP}}
\newcommand{\CC}{\mathbb{C}}

\newcommand{\RR}{\mathbb{R}}

\newcommand{\U}{{\rm{U}}}

\newcommand{\z}{\mathbf{z}}
\numberwithin{equation}{section}
\numberwithin{theorem}{section}
\numberwithin{table}{section}
\numberwithin{table}{section}
\begin{document}
\bibliographystyle{amsalpha} 
\title{A sm\"org\r{a}sbord of scalar-flat K\"ahler ALE surfaces}
\author{Michael T. Lock}
\thanks{The first author was partially supported by NSF Grant DMS-1148490}
\address{Department of Mathematics, University of Texas, Austin, 
TX, 78712}
\email{mlock@math.utexas.edu}
\author{Jeff A. Viaclovsky}
\address{Department of Mathematics, University of Wisconsin, Madison, 
WI, 53706}
\email{jeffv@math.wisc.edu}
\thanks{The second author was partially supported by NSF Grant DMS-1405725}

\dedicatory{This article is dedicated to the memory of Egbert Brieskorn.}
\begin{abstract} There are many known examples of scalar-flat
K\"ahler ALE surfaces, all of which have 
group at infinity either cyclic or contained in ${\rm{SU}}(2)$. 
The main result in this paper shows that for {\textit{any}} non-cyclic finite subgroup
$\Gamma \subset \U(2)$ containing no complex reflections, there exist
scalar-flat K\"ahler ALE metrics on the minimal 
resolution of  $\CC^2 / \Gamma$, for which $\Gamma$ occurs as the 
group at infinity. Furthermore, we show that these metrics admit a holomorphic isometric circle action.
It is also shown that there exist scalar-flat K\"ahler ALE
metrics with respect to some small deformations of complex structure
of the minimal resolution. Lastly, we show the existence of  
extremal K\"ahler metrics admitting holomorphic isometric circle actions 
in certain K\"ahler classes on the complex analytic 
compactifications of the minimal resolutions. 
\end{abstract}
\maketitle
%\parskip1pt
%\setcounter{tocdepth}{1}
%\tableofcontents
%\vspace{-10mm}
%%%%%%%%%%%%%%%%%%%%%%%%%%%%%%%%%%%%%%%%%%%%%%%%%%%%%%%%
\section{Introduction}
\label{intro}
%%%%%%%%%%%%%%%%%%%%%%%%%%%%%%%%%%%%%%%%%%%%%%%%%%%%%
In order to state the main theorem, we begin with some 
preliminary definitions. Let $\Gamma \subset \U(2)$ be a finite subgroup. 
Then the following notions are equivalent. 
\begin{itemize}
\item $\Gamma$ contains no complex reflections.

%\vspace{2mm}
\item
$\Gamma$ acts freely on $S^3$

%\vspace{2mm}
\item
$X = \CC^2 / \Gamma$ has an isolated singularity at the origin.
\end{itemize}
While an in depth discussion of these groups is given in Section \ref{U(2)_subgroups}, for the sake of clarity here, we provide a 
list of all possible such groups below in Table~\ref{groups}.  However, before doing so, it is necessary to introduce some notation.  

\begin{itemize}
\item 
For $q$ and $p$ relatively prime integers, $L(q,p)$ denotes the cyclic subgroup of ${\rm{U}}(2)$ generated by
\begin{align*}
\left(
\begin{matrix}
\exp(2\pi i/p)&0\\
0&\exp(2\pi i q /p)
\end{matrix}
\right).
\end{align*}

%\vspace{2mm}
\item
The binary polyhedral groups (dihedral, tetrahedral, octahedral, icosahedral) are respectively denoted by $D^*_{4n}$, $T^*$, $O^*$, $I^*$. 

%\vspace{2mm}
\item 
The map $\phi:{\rm{SU}}(2) \times  {\rm{SU}}(2) \rightarrow {\rm SO}(4)$ denotes the usual double cover (see
\eqref{phi} below).
\end{itemize}

\begin{table}[ht]
\caption{Finite subgroups of \U(2) containing no complex reflections}
\label{groups}
\begin{tabular}{ll l}
\hline
$\Gamma\subset{\rm U}(2)$ & Conditions & Order\\
\hline\hline
$L(-m,n)$   & $(m,n) = 1$ & $n$\\
$ \phi(L(1,2m) \times D^*_{4n}) $ &  $(m,2n) = 1$ & $4mn$\\
$ \phi(L(1,2m) \times T^*) $  & $(m,6) = 1$ & $24m$\\
$  \phi(L(1,2m) \times O^*) $ & $(m,6) = 1$ & $48m$\\
$  \phi(L(1,2m)\times I^*) $ & $(m,30) = 1$ & $120m$ \\
 Index--$2$ diagonal $\subset   \phi(L(1,4m)\times D^*_{4n})$&  $(m,2) = 2,(m,n)=1$& $4mn$\\
 Index--$3$ diagonal $\subset  \phi(L(1,6m) \times T^*)$  & $(m,6)=3$ & $24m$.\\
\hline
\end{tabular}
\end{table}

\begin{definition}
\label{minresdef}
{\em Let $\Gamma \subset \U(2)$ be a finite subgroup containing no complex reflections.  Then, a smooth complex surface $\tilde{X}$ is called a {\textit{minimal resolution}} 
of $\CC^2/\Gamma$ if there is a mapping $\pi : \tilde{X}\rightarrow \CC^2/\Gamma$ such 
that 
\begin{enumerate}
\item The restriction $\pi:\tilde{X} \setminus \pi^{-1}(0)\rightarrow \CC^2/\Gamma\setminus\{0\}$ is a biholomorphism;
\item $\pi^{-1}(0)$ is a divisor in $\tilde{X}$ containing no $-1$ curves.
\end{enumerate}}
\end{definition}
These resolutions are {\textit{minimal}} in the sense that, given any other resolution $\pi_Y :Y \rightarrow \CC^2/\Gamma$, there is a proper analytic map $p:Y \rightarrow \tilde{X}$ such that $\pi_Y= \pi\circ p$.  For each such $\Gamma$, up to isomorphism there exists a unique such minimal resolution, and in 1968, Brieskorn completely described these resolutions 
complex-analytically \cite{Brieskorn}.  The exceptional divisor consists of
a tree of rational curves with normal crossing singularities. 
For each non-cyclic subgroup in Table \ref{groups}, it turns out that the exceptional divisor contains a distinguished rational curve 
which intersects exactly three other rational curves. This curve will be referred to as 
the {\textit{central rational curve}}. The self-intersection number of this 
curve will be denoted $-b_{\Gamma}$. The signature $\tau(\tilde{X})$ is 
given by $-k_{\Gamma}$, where $k_{\Gamma}$ is the total number of rational 
curves in the exceptional divisor. The exceptional divisor is described by a
tree with three branches attached to the central rational curve, and each branch
corresponds to a Hirzebruch-Jung string. We will denote the corresponding graph 
by
\begin{align}
\label{bnot}
\langle -b_{\Gamma}; (\alpha_1, \beta_1) ;(\alpha_2, \beta_2); (\alpha_3, \beta_3) \rangle,
\end{align}
where $\alpha_i$ and $\beta_i$, for $i = 1, 2, 3$,
are integers determined by the group $\Gamma$,  see Section~\ref{quotients_g_LB}. 

For historical context it is important to note the work of Seifert \cite{Seifert}, on three-dimensional fibered spaces, and Seifert-Threlfall \cite{Threlfall-Seifert_1,Threlfall-Seifert_2}, on three-dimensional spherical space forms, which was essential to that of Brieskorn.  Also, for a further exposition on quotient singularities of a complex surface, we refer the reader to \cite{Behnke-Riemenschneider}.

We are interested in certain metric 
structures on these spaces. Let $(M,g,J)$ be a K\"ahler manifold with K\"ahler form 
$\omega = g(J \cdot, \cdot)$.  Critical points of 
the functional 
\begin{align}
\alpha \mapsto \int_M R_{\alpha}^2 dV_{\alpha},
\end{align}
where $\alpha$ is a K\"ahler form in the cohomology class $[\omega]$ 
and $R_{\alpha}$ is the scalar curvature of the metric $g_{\alpha}$ determined by $\alpha$,
are called {\textit{extremal K\"ahler metrics}}.  These were introduced 
by Calabi in \cite{Calabi1, Calabi2}. The Euler-Lagrange 
equation for a critical point is given by 
\begin{align}
\label{E-L_extremal}
\overline{\partial} \partial_g^{\#}R_{\alpha} = 0,
\end{align}
where the operator $\partial_g^{\#}$ denotes the $(1,0)$ part of the
gradient of a function with respect to the metric $g_{\alpha}$. In other words, the 
$(1,0)$ part of the gradient of the scalar curvature is a 
holomorphic vector field. 

In particular, note that scalar-flat 
K\"ahler metrics are extremal. A very important class of scalar-flat
K\"ahler metrics are those of the following type.
\begin{definition}
\label{ALEdef}
{\em
 A complete Riemannian manifold $(X^4,g)$ 
is called {\textit{asymptotically locally 
Euclidean}} ({\textit{ALE}}) of order $\tau$ if 
there exists a finite subgroup 
$\Gamma \subset {\rm{SO}}(4)$ 
which acts freely on $S^3$, a compact subset $K\subset X$, and a diffeomorphism
\begin{align} 
\Psi : X \setminus K \rightarrow ( \mathbb{R}^4 \setminus B(0,R)) / \Gamma,
\end{align}
such that under this identification 
\begin{align}
(\Psi_* g)_{ij} &= \delta_{ij} + O( \rho^{-\tau}),\\
\ \partial^{|k|} (\Psi_*g)_{ij} &= O(\rho^{-\tau - k }),
\end{align}
for any partial derivative of order $k$, as
$\rho \rightarrow \infty$, where $\rho$ is the distance to some fixed basepoint.  We call $\Gamma$ the {\textit{group at infinity}}.  }
\end{definition}

We next review briefly the known examples of scalar-flat K\"ahler ALE surfaces:
\begin{itemize}
\item
In \cite{LeBrunnegative}, for all positive integers $n$, LeBrun constructed a ${\rm U}(2)$-invariant scalar-flat K\"ahler ALE metric on the total space of the bundle $\mathcal{O}(-n)$ over $\CP^1$, the minimal resolution of $\CC^2/L(1,n)$.  The $n=1$ and $n=2$ cases are the well-known Burns and Eguchi-Hanson metrics respectively \cite{Burns, EguchiHanson}. 

\item
The cases in Table \ref{groups} where $m= 1$ are the ADE-type singularities, and 
hyperk\"ahler metrics on these minimal resolutions were produced and classified by Kronheimer using the hyperk\"ahler quotient construction \cite{Kronheimer, Kronheimer2}.  These
generalize earlier examples of Eguchi-Hanson and Gibbons-Hawking, see \cite{EguchiHanson,GibbonsHawking} respectively.

\item Using a construction of Joyce, examples of scalar-flat K\"ahler ALE metrics 
with a torus action for the general cyclic case $L(m,n)$ were given by 
Calderbank-Singer \cite{CalderbankSinger}. 
\end{itemize}

Our main result shows that {\textit{any}} allowed group may 
occur as the group at infinity of a scalar-flat K\"ahler ALE surface:
\begin{theorem}
\label{t1}
Let $\Gamma \subset \U(2)$ be a finite subgroup containing no complex
reflections. 
Then the minimal resolution of $\CC^2 / \Gamma$ admits 
scalar-flat K\"ahler ALE metrics. Furthermore, these metrics 
admit a holomorphic isometric circle action. 
\end{theorem}
The proof of Theorem \ref{t1} is given in Section \ref{SFKsec}, 
and the basic idea is as follows.  First, we take quotients of the LeBrun negative mass metrics on $\mathcal{O}(-\ell)$ (for a certain value of $\ell$ 
determined by $\Gamma$) by certain finite 
subgroups of ${\rm{U}}(2)$.  This results in scalar-flat K\"ahler ALE orbifolds which have only
isolated singularities of cyclic type.  Then, we resolve these singularities by adapting a gluing theorem for Hermitian anti-self-dual metrics, due to Rollin-Singer, to the situation where the base metric is an ALE orbifold (instead of a compact orbifold), and use this to attach appropriate Calderbank-Singer spaces.  
Finally, a theorem of Boyer shows that these Hermitian anti-self-dual metrics are in fact scalar-flat K\"ahler.
To find the holomorphic isometric circle action, we show that our approximate metrics 
can be chosen to have such an action, and then use an equivariant 
version of the gluing theorem. 
We note that this gluing procedure can be thought of as a metric version of 
Brieskorn's method of constructing the minimal resolution complex analytically. 

 We remark that, using a recent result of Apostolov-Rollin \cite{Apostolov-Rollin}, 
we can produce examples of scalar-flat K\"ahler ALE metrics with any cyclic group 
at infinity starting with only the canonical Bochner-K\"ahler metrics on certain weighted projective spaces. This is discussed and proved in Section \ref{BK_WPS} and 
Section \ref{cyclic_construction}. Consequently, the main result in 
Theorem \ref{t1} can be obtained using only these Bochner-K\"ahler metrics, 
without using any knowledge of the Calderbank-Singer construction. 
These Bochner-K\"ahler metrics can therefore be viewed as the fundamental 
building blocks of this entire construction.

\subsection{Deformations of complex structure}
In the cases where $\Gamma \subset {\rm{SU}}(2)$, 
the hyperk\"ahler quotient construction
produces hyperk\"ahler metrics for the minimal
resolution complex structure as well as for all small deformations of the minimal resolution complex structure.
In the case of the LeBrun negative mass metrics, Honda has shown that all small 
deformations of the complex structure on $\mathcal{O}(-n)$ also admit 
scalar-flat K\"ahler metrics \cite{HondaOn, Honda_2014}. 
For our next result, we 
obtain scalar-flat K\"ahler metrics for {\textit{some}} small deformations of the minimal resolution complex structure for the non-cyclic subgroups.
\begin{theorem}
\label{t2}
If $\Gamma \subset {\rm{U}}(2)$ is a non-cyclic finite subgroup containing no complex
reflections, then
some small deformations of the minimal resolution of $\CC^2 / \Gamma$ admit
scalar-flat K\"ahler ALE metrics. The space of these deformations 
has real dimension $2 ( b_{\Gamma} - 1)$, where $-b_{\Gamma}$ is the self-intersection
of the central rational curve. For each of these deformations of complex structure, 
there exist scalar-flat K\"ahler ALE metrics for an open set in the K\"ahler cone
of K\"ahler classes. Consequently, there is a family of scalar-flat K\"ahler ALE metrics of real dimension at least 
\begin{align}
\label{mdimform}
2(b_{\Gamma} -1) + k_{\Gamma},
\end{align}
where $k_{\Gamma}$ is
the total number of rational curves in the exceptional divisor. 
\end{theorem}
In the cyclic case, however, it is not known 
whether there are scalar-flat K\"ahler ALE metrics for deformations 
of the minimal resolution complex structure, 
except of course when these overlap with the aforementioned 
cases (toric multi-Eguchi-Hanson metrics and LeBrun negative mass metrics).

In general, for the minimal resolution $\tilde{X}$ of $\mathbb{C}^2/ \Gamma$,  the sheaf cohomology group $H^1(\tilde{X},\Theta_{\tilde{X}})$ has dimension
\begin{align}
\label{dimform}
\dim_{\mathbb{C}}( H^1(\tilde{X}, \Theta_{\tilde{X}})) = \sum_{i =1}^{k_{\Gamma}} ( e_i -1),
\end{align}
where $-e_i$ is the self-intersection number of the $i$th rational curve. 
So the examples in Theorem \ref{t2} are only for deformations corresponding to a proper subspace of $H^1(\tilde{X}, \Theta_{\tilde{X}})$. However, there seems to be no obstruction for other
small deformations of complex structure to admit scalar-flat K\"ahler ALE metrics, so we make the following conjecture. 
\begin{conjecture} For any finite subgroup $\Gamma \subset {\rm{U}}(2)$ containing no complex
reflections, all small deformations of complex structure of the minimal resolution of $\CC^2/\Gamma$ admit scalar-flat K\"ahler ALE metrics.
\end{conjecture}
To approach the conjecture, an approach similar to that of \cite{BiquardRollin}, 
in which one deforms the complex structure, is likely needed, although in this 
case one needs to moreover completely understand the 
scalar-flat K\"ahler deformations of the Calderbank-Singer metrics. 

\subsection{The hyperk\"ahler case}
We point out that for any non-cyclic $\Gamma \subset {\rm{SU}}(2)$, our construction does actually 
yield hyperk\"ahler metrics with group $\Gamma$ at infinity, see Proposition \ref{Kronheimer_sfk_gluing}. Thus, a corollary is that for any such subgroup, {\textit{some}} 
hyperk\"ahler metrics can be obtained by gluing techniques. 
Furthermore, the only ingredients needed in this case are the Gibbons-Hawking multi-Eguchi-Hanson
metrics. 

As we will see below in Section~\ref{hyperkahler}, using a gluing theorem for anti-self-dual metrics 
(instead of the Hermitian anti-self-dual gluing theorem), we can actually show that {\textit{all}} 
small deformations also admit hyperk\"ahler metrics, and hence 
obtain an open set in the moduli space. Once again, these metrics
are known by Kronheimer's construction, but we find it interesting that
they can be obtained by gluing techniques, using only the 
multi-Eguchi-Hanson metrics as building blocks. 

\subsection{Extremal K\"ahler metrics on rational surfaces} 
 In  \cite[Lemma 4.1]{LeBrunMaskit}, 
LeBrun-Maskit showed that any scalar-flat K\"ahler ALE space
admits a complex analytic compactification to a rational surface 
by adding a tree of rational curves at infinity.
In Section \ref{ca_compactifications}, we identify the surfaces obtained from 
compactifying the spaces obtained in Theorem \ref{t1}. More precisely, we
will show that a compactification can be obtained by
adding the tree of rational curves
\begin{align}
\label{bnot2}
\langle b_{\Gamma}'; (\beta_i - \alpha_1, \beta_1) ;(\beta_2 - \alpha_2, \beta_2); 
(\beta_3 - \alpha_3, \beta_3) \rangle,
\end{align}
where $b_{\Gamma}' > 0$. While the existence of such compactifiations is known, for the sake of completeness, we provide a short proof using weighted projective spaces.
Here, we will let $k_i$ denote the length of the Hirzebruch-Jung algorithm for $L(\alpha_i,\beta_i)$
and $\ell_i$ denote the length of the Hirzebruch-Jung algorithm for $L(\beta_i - \alpha_i,\beta_i)$
for $i = 1, 2, 3$  
(see Section \ref{building_blocks} for a description of this algorithm). 
\begin{proposition} 
\label{cac1intro}
Let $\Gamma \subset \U(2)$ be a non-cyclic finite subgroup containing no complex
reflections. 
Then the minimal resolution of $\CC^2 / \Gamma$ has a complex 
analytic compactification diffeomorphic to $\CP^2 \# \mathfrak{k} \overline{\CP}^2$,
where $\mathfrak{k} = k_1 + k_2 + k_3 + \ell_1 + \ell_2 + \ell_3 +1$.
\end{proposition}
The complete list of rational curves in the compactification is given by
\begin{align}
&[b_{\Gamma}'], [-b_{\Gamma}],\\
&[-b^i_\ell],  i = 1, 2, 3, \ \ell = 1 \dots \ell_i\\
&[-e^i_k],  i = 1,2,3, \ k = 1 \dots k_i,
\end{align}
where notation $[q]$ means that $q$ is the self-intersection 
number of the corresponding rational curve. 

Next, we have an existence result for extremal K\"ahler metrics 
in certain K\"ahler classes on the complex analytic compactifications
found in Proposition \ref{cac1intro}.
\begin{theorem}
\label{cacintro}
There exists an $\varepsilon_0 > 0$ such that for 
$\varepsilon < \varepsilon_0$, and for constants
\begin{align*}
&a_1 > 0,  a_2 > 0,\\
&a^i_{\ell} > 0, i = 1, 2, 3, \ \ell = 1 \dots \ell_i,\\  
&c^i_{k} > 0 , i = 1,2,3, \ k = 1 \dots k_i,
\end{align*}
the compactification in Proposition \ref{cac1intro} admits extremal K\"ahler metrics
with non-constant scalar curvature in the 
K\"ahler class
\begin{align*}
a_1 [b'] + a_2 [-b] 
+ \sum_{i=1}^3 \sum_{\ell = 1}^{\ell_i} f(i, \varepsilon) a^i_{\ell}[b^i_{\ell}]
+ \sum_{i=1}^3 \sum_{k = 1}^{k_i} f'(i, \varepsilon)c^i_{k} [e^i_{k}],
\end{align*}
where
\begin{align*}
f(i, \varepsilon)=
\begin{cases}
\varepsilon^2 & \alpha_i \neq 1 \\
\varepsilon^4 &  \alpha_i = 1,
\end{cases}
\phantom{=} \mbox{and} \phantom{=}f'(i, \varepsilon)=
\begin{cases}
\varepsilon^2 & \alpha_i \neq \beta_i- 1 \\
\varepsilon^4 & \alpha_i = \beta_i - 1.
\end{cases}
\end{align*}
\end{theorem}
The idea of the proof of Theorem \ref{cacintro} is similar to 
that of Theorem \ref{t1}, but 
instead relies on gluing theorems for extremal K\"ahler metrics 
due to Arezzo-Lena-Mazzieri \cite{ALM_resolution} and Arezzo-Pacard-Singer \cite{APSinger}, see also \cite{API, APII, GaborI, GaborII,ALM_kummer}.
The introduction of the functions $f$ and $f'$ is necessary 
because if one of the bubbles is Ricci-flat, then it must 
be glued on with a different scaling. This happens because 
in general scalar-flat K\"ahler ALE surfaces are ALE of
order $2$ \cite{AV12, Streets}, but if such a space is moreover Ricci-flat, then 
it is ALE of order $4$ \cite{BKN, ct}.

\begin{remark}{\em
 The case with minimal $\mathfrak{k}$ is the case of $\Gamma = D^*_8$ (the quaternion group), for which Theorem 
\ref{cacintro} produces extremal K\"ahler metrics on $\CP^2 \# 7 \overline{\CP}^2$.
The next smallest cases are $\Gamma = D^*_{12}$ or $\Gamma = \mathcal{I}^2_{2,3}$, 
which yield examples on $\CP^2 \# 8 \overline{\CP}^2$.
}
\end{remark}
We refer the reader to \cite{ChiLi, ConlonHeinIII, Suvaina} for other results and examples 
regarding complex analytic compactifications.

We conclude this section with a corollary of our work here combined with that of Arezzo-Lena-Mazzieri \cite{ALM_resolution}.  In a sense, the following is a more generalized and less specific version of Theorem \ref{cacintro}.
\begin{theorem}
Let $(M,g)$ be a compact extremal K\"ahler orbifold surface.  Then the minimal resolution $\widehat{M}$ of $M$ admits extremal K\"ahler metrics in certain K\"ahler classes.  
\end{theorem}
The proof follows directly from using \cite[Theorem 1.1]{ALM_resolution}, together with the scalar-flat K\"ahler ALE metrics from Theorem \ref{t1} above.   See \cite{ALM_resolution} for a precise description of these K\"ahler classes, analogous to that of Theorem \ref{cacintro}.

\subsection{Acknowledgements}
The authors would like to thank Nobuhiro Honda and Claude LeBrun for many 
enlightening discussions on scalar-flat K\"ahler metrics. 
We would also like to thank Claudio Arezzo, David Calderbank, and Michael Singer for 
many helpful comments.  Also, discussions with Akira Fujiki, 
Ryushi Goto, and Hans-Joachim Hein  
on deformations of complex structure were very helpful.  Finally, thanks are given to the anonymous referees for numerous suggestions to improve the exposition of the paper.

%%%%%%%%%%%%%%%%%%%%%%%%%%%%%%%%%%%%%%%%%%%%%%%%%%%%
\section{Building blocks}
\label{building_blocks}

Scalar-flat K\"ahler ALE metrics with cyclic groups at infinity are essential to the work of this paper because, in a sense, they are the ``building blocks'' of our construction.  We introduce several important families here.  
In Section \ref{LB_metrics} and Section~\ref{CS_metrics}, we provide brief descriptions of the LeBrun negative mass metrics and the Calderbank-Singer metrics respectively.  The former family of metrics is, in fact, contained in the latter, however we introduce and discuss these negative mass metrics separately, as well as outline their original construction, because it will be necessary to understand the action of the isometry group on these spaces.  Also, in Subsection~\ref{CS_metrics}, we discuss the topology of the minimal resolutions of cyclic quotients of $\CC^2$ on which these metrics lie.  In Section \ref{M-E-H_metrics}, we introduce the Gibbons-Hawking multi-Eguchi-Hanson metrics.  The subclass of these given by those with a toric symmetry is actually contained in the family of Calderbank-Singer metrics, so their introduction is not strictly necessary.  However, we discuss them separately in light of the results concerning the Kronheimer hyperk\"ahler metrics described in Section~\ref{hyperkahler} below.  Finally, in Section \ref{BK_WPS}, we introduce the canonical Bochner-K\"ahler metrics on weighted projective spaces, as well as a conformal relationship between these and scalar-flat K\"ahler ALE orbifolds.  These metrics on weighted projective spaces will be seen to be the ``fundamental'' building blocks of our construction; in Section \ref{cyclic_construction} below, we will show how using only these, it is possible to obtain examples of the other building blocks discussed in Sections~\ref{LB_metrics}--\ref{M-E-H_metrics}.

\subsection{LeBrun negative mass metrics}
\label{LB_metrics}
In \cite{LeBrunnegative}, for all positive integers $n$, LeBrun constructed a ${\rm U}(2)$-invariant scalar-flat K\"ahler ALE metric on the total space of the bundle $\mathcal{O}(-n)$ over $\CP^1$, the minimal resolution of $\CC^2/L(1,n)$.  The $n=1$ and $n=2$ cases are the well-known Burns and Eguchi-Hanson metrics respectively \cite{Burns, EguchiHanson}.  When $n>2$, LeBrun showed that these metrics have negative mass thus providing an infinite family of counterexamples to the generalized positive action conjecture \cite{HawkingPope}.  We next
briefly describe the method of construction from \cite{LeBrunnegative}.

Since any 
K\"ahler metric satisfies $R\omega\wedge\omega=\rho\wedge\omega$, where $R$ is the scalar curvature and $\rho$
is the Ricci form, he examined the equation
\begin{align}
\label{R=0}
0=\rho\wedge\omega.
\end{align}
Restricting to those metrics on $\CC^2\setminus\{0\}$ which result from a ${\rm U}(2)$-invariant K\"ahler potential,
i.e. those having K\"ahler forms $\omega=\sqrt{-1}\partial\bar{\partial}\phi$ with potential functions $\phi(\z)$, where $\z=|z_1|^2+|z_2|^2$, on
$\CC^2\setminus\{0\}$, LeBrun extracted a nonlinear ODE from this which he then solved to obtain a family of ${\rm U}(2)$-invariant potential functions $\{\phi_{n}(\z)\}_{n\in\mathbb{Z}^+}$ for which \eqref{R=0} is satisfied.
For each $\phi_{n}(\z)$,  after defining a new radial coordinate 
 \begin{align}
\label{r-coordinate}
 r=\sqrt{\z\frac{\partial\phi_{n}}{\partial \z}},
 \end{align}
 LeBrun showed that the corresponding metric is
\begin{align}
\label{negativemassmetric}
g_{LB}=\frac{dr^2}{1+\frac{n-2}{r^2}+\frac{1-n}{r^4}}+r^2\Big[\sigma_1^2+\sigma_2^2+\Big(1+\frac{n-2}{r^2}+\frac{1-n}{r^4}\Big)\sigma_3^2\Big],
\end{align}
where $r$ is the radial distance from the origin and $\sigma_1,\sigma_2, \sigma_3$ are the usual left-invariant coframe on ${\rm SU}(2)=S^3$.
This metric is scalar-flat K\"ahler, hence anti-self-dual, but has an apparently singularity at $r=1$.  
However, redefine the radial coordinate as $\tilde{r}^2~=~(r^2-1)$ and attaching a $\CP^1$ at $\tilde{r}=0$, and then take the $\mathbb{Z}_n$ quotient of $\CC^2\setminus \{0\}$ generated by
\begin{align}
\label{Znaction}
(z_0,z_1)\mapsto(e^{2\pi i/n}z_0,e^{2\pi i/n}z_1).
\end{align}
The metric now extends smoothly over the $\CP^1$ at $\tilde{r}=0$, and therefore $g_{LB}$ defines a ${\rm{U}}(2)$-invariant scalar-flat K\"ahler ALE  metric on the total space of the bundle $\mathcal{O}(-n)$ over $\CP^1$.  Clearly, the group at infinity is $L(1,n)$.
 
\subsection{Calderbank-Singer metrics}
\label{CS_metrics}
For every pair of relatively prime integers $1\leq q<p$, Calderbank-Singer constructed a family of scalar-flat K\"ahler, hence anti-self-dual, ALE metrics on the minimal resolution of $\CC^2/L(q,p)$, the topology of which is discussed below, \cite{CalderbankSinger}.  For $1<q<p$ these metrics are toric and come in families of dimension $k-1$.  For $q=1$ and $q=p-1$, these metrics are the LeBrun negative mass metrics
and the toric multi-Eguchi-Hanson metrics respectively, see Remark \ref{toric_GH}.

They used the 
Joyce ansatz \cite{Joyce1995} to find these metrics explicitly, the idea of which is as follows.  Let $(B,h)$ be a spin $2$-manifold with constant curvature $-1$, $W\rightarrow B$ be the spin-bundle with the induced metric (denoted by $h$ as well), and $V$ be a two-dimensional vector space with symplectic structure $\varepsilon(\cdot, \cdot)$.  Joyce defined a $V$-invariant metric on the product bundle $B\times V$ associated to each bundle isomorphism $\Phi:W\rightarrow B\times V$.  First, he defines a family of metrics on $V$ by
\begin{align}
(v,\tilde{v})_{\Phi}=h(\Phi^{-1}(v),\Phi^{-1}(\tilde{v})).
\end{align}
Notice that this family is parameterized by $B$.  Then, he defines a metric on the total space by
\begin{align}
g_{\Phi}=\Omega^2(h+(\cdot, \cdot)_{\Phi}),
\end{align}
for a nonvanishing conformal factor $\Omega\in C^{\infty}(B)$.  From the construction $g_{\Phi}$ is clearly invariant under the action of $V$ on $B\times V$, so it descends to the quotient $B\times (V/\Lambda)$ where $\Lambda\subset V$ is any lattice.  Joyce then showed that this metric is (anti-)self-dual if $\Phi$ satisfies a certain linear differential equation.  Calderbank-Singer considered $B\subset \mathcal{H}^2$, for which Joyce showed that there are
scalar-flat K\"ahler representative in the conformal class of each
 $g_{\Phi}$, and studied the solutions of the Joyce equations.  
Combining this with work on toric $4$-manifolds, see \cite{OrlikRaymond}, Calderbank-Singer explicitly completed their construction.

\begin{remark}
{\em
In \cite{Wright}, Wright showed that if $(X,g)$ is a toric scalar-flat K\"ahler ALE space with group at infinity $\Gamma$ satisfying $\chi_{orb}(X)>-2/|\Gamma|$, where $\chi_{orb}$ is the orbifold Euler characteristic (see \eqref{chi_orb}), then it is isometric to a Calderbank-Singer metric on the minimal resolution of $\CC^2/\Gamma$.  Clearly this inequality is satisfied whenever the underlying space is such a minimal resolution, so the Calderbank-Singer spaces are a complete list of the {\em{toric}} scalar-flat K\"ahler ALE metrics on minimal resolutions.
}
\end{remark}

The minimal resolutions of cyclic quotients of $\CC^2$ are classical \cite{Hirzebruch1953}.  In order to describe these, we introduce the following modified Euclidean algorithm.  For relatively prime integers $1 \leq q < p$, write
\begin{align}
\begin{split}
\label{modified_EA}
p&=e_1q-a_1\\
q&=e_2a_1-a_2\\
&\hspace{2mm} \vdots \\
a_{k-3}&=e_{k-1}a_{k-2}-1\\
a_{k-2} &= e_k a_{k-1} = e_k,
\end{split}
\end{align}
where $e_i \geq 2$, and $0 \leq a_i < a_{i-1} < \dots < a_1 < q,$ for $i = 1, \cdots ,k$.  We refer to the integer $k$ as the {\textit{length}} of the modified Euclidean algorithm. This can 
also be written as the continued fraction expansion
\begin{align}
\frac{q}{p} = \cfrac{1}{e_1 - \cfrac{1}{e_2 - \cdots \cfrac{1}{e_k}}}.
\end{align}
where the $e_i$ and $k$ are as in \eqref{modified_EA}.  

Now,  let $\tilde{X}$ be the minimal resolution 
of $\CC^2/ L(q,p)$, where $1\leq q<p$ are relatively prime integers.
The intersection matrix of the exceptional divisor $\tilde{X}$ is shown in 
Figure \ref{cycfig}, where the $e_i$ and $k$ are defined above with $e_i \geq 2$.

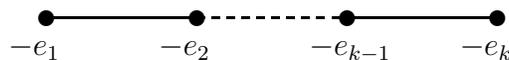
\begin{figure}[h]
\setlength{\unitlength}{2cm}
\begin{picture}(7,0.1)(-1,0)
\linethickness{.3mm}
\put(1,0){\line(1,0){1}}
\put(1,0){\circle*{.1}}
\put(.75,-.25){$-e_1$}
\put(2,0){\circle*{.1}}
\put(1.75, -.25){$-e_2$}
\multiput(2,0)(0.1,0){10}
{\line(1,0){0.05}}
\put(3,0){\circle*{.1}}
\put(2.75, -.25){$-e_{k-1}$}
\put(3,0){\line(1,0){1}}
\put(4,0){\circle*{.1}}
\put(3.75, -.25){$-e_k$}
\end{picture}
\vspace{5mm}
\label{cycfig}
\caption{Exceptional divisor in the Hirzebruch-Jung resolution.}
\end{figure}
\begin{remark}{\em
%If $p>0$, let $\widetilde{q}^{(p)}$ denote $(q\text{ mod } p)$, so $1\leq \widetilde{q}^{(p)}<p$ and $L(q,p)=L(\widetilde{q}^{(p)},p)$.
For a group $L(q,p)$, where $q$ and $p$ are relatively prime integers with $1 \leq q \leq p-1$, it will frequently be necessary to consider the length of this modified Euclidean algorithm, which we will denote by $k_{L(q,p)}$.
}
\end{remark}

\subsection{Multi-Eguchi-Hanson metrics}
\label{M-E-H_metrics}
In \cite{GibbonsHawking}, for all positive integers $n$, Gibbons and Hawking constructed hyperk\"ahler ALE metrics with cyclic group action $L(-1,n)$ at infinity.  
These are known as multi-Eguchi-Hanson metrics, and a brief description of the construction is as follows.  

Choose $n$ distinct (monopole) points $P=\{p_1,\cdots, p_n\}$ in $(\RR^3,g_{Euc}),$ and let $G_{p_i}$ be the fundamental solution of the Laplacian based at $p_i$ with normalization so that $\Delta G_{p_i}=2\pi \delta_{p_i}$.  Then, consider the function $V=\frac{1}{2}\sum_{i=1}^nG_{p_i}$.  Since this 
is harmonic on $\RR^3\setminus P$, the $2$-form $*dV$ is closed here and $\frac{1}{2\pi}[*dV]\in H^2(\RR^2\setminus P,\mathbb{Z})$.  Let $X_0\rightarrow \RR^3\setminus P$ be the unique  principal ${\rm U}(1)$-bundle corresponding to this cohomology class.  By Chern-Weil theory there is a connection form $\omega$ on $X_0$ with curvature $d\omega=*dV$.  The Gibbons-Hawking metric is defined on $X_0$ by
\begin{align}
g_{GH}=Vg_{Euc}+V^{-1}\omega^2.
\end{align}
Now, define a larger manifold $X$ by attaching points $\tilde{p_i}$ over each $p_i$, over which $g_{GH}$ extends to a smooth Riemannian metric.  The space $(X,g_{GH})$ is hyperk\"ahler ALE with group at infinity $L(-1,n)$, and is clearly $S^1$-equivariant. 

\begin{remark}
\label{toric_GH}
{\em When all of the monopole points lie on a common line, these metrics are toric and, in fact, are exactly the Calderbank-Singer metrics for $q=p-1$.  This construction can also have a multiplicity at each point which results in orbifold metrics, see \cite{ViaclovskyFourier}.
}
\end{remark}

\subsection{Weighted projective spaces}
\label{BK_WPS}
Lastly, we introduce the canonical Bochner-K\"ahler metrics on weighted projective spaces.  In a sense, these will be our ``fundamental building blocks'' as we will eventually use them to obtain the necessary ``building blocks,'' discussed above, in a uniform way.
\begin{definition} {\em{ For relatively prime integers $1 \leq r \leq q \leq p$, 
the {\em{weighted projective space}} $\mathbb{CP}^2_{(r,q,p)}$ is $S^{5}/S^1$, 
where $S^1$ acts by
\begin{align*}
(z_0,z_1,z_2)\mapsto (e^{ir\theta}z_0,e^{iq\theta}z_1 ,e^{ip\theta}z_2),
\end{align*}
for $0\leq \theta <2\pi$.
}}
\end{definition}
Weighted projective spaces are complex orbifold surfaces.  For relatively prime positive integers $a$ and $b$, let $a^{-1;b}$ denote the inverse of $a$ modulo $b$, i.e., $1\leq a^{-1;b}<b$ and $a^{-1;b}a\equiv 1\mod b$.  With the standard orientation, $\mathbb{CP}^2_{(r,q,p)}$ has the following singular points and corresponding orbifold groups.
\begin{align}
\label{WPS_table}
\begin{array}{lll}
\hline
\text{Singularity} & &\text{Orbifold group} \\\hline\hline
\text{$[1,0,0]$}&\phantom{=}&L(q^{-1;r}p,r)\\
\text{$[0,1,0]$} &\phantom{=}&  L(p^{-1;q}r,q)\\
\text{$[0,0,1]$}&\phantom{=}& L(r^{-1;p}q,p)\\
\hline
\end{array}
\end{align}
Note that there can be one, two or three singular points depending on how many of the integers $r,q,p$ are strictly greater than $1$.

Bryant showed that every weighted projective space admits a Bochner-K\"ahler metric \cite{Bryant}.  Later, in \cite{DavidGauduchon}, David-Gauduchon gave a construction of such a metric on each weighted projective space and used an argument, due to Apostolov, to prove that it was in fact the unique Bochner-K\"ahler metrics on that space \cite[Appendix~D]{DavidGauduchon}.  Hence, we will refer to these metrics as the {\em{canonical Bochner-K\"ahler metrics}}, and denote them by $(\CP^2_{(r,q,p)},g_{BK})$.  On complex surfaces the Bochner tensor is exactly the anti-self-dual part of the Weyl curvature tensor, so these metrics are self-dual K\"ahler metrics.

In  Section \ref{cyclic_construction} below, we will use the canonical Bochner-K\"ahler metrics on certain weighted projective spaces, along with conformal results of \cite{Joyce1991,DabkowskiLock,Apostolov-Rollin} and Theorem \ref{SFK_gluing_theorem}, the scalar-flat K\"ahler ALE gluing theorem, to obtain examples of scalar-flat K\"ahler ALE metrics on minimal resolutions of cyclic quotient singularities in a uniform inductive manner.  Subsequently, we will not only prove the existence of scalar-flat K\"ahler ALE metrics on every possible minimal resolution, but see that the only building blocks needed are the canonical Bochner-K\"ahler metrics on certain weighted projective spaces.

We conclude this section with a discussion of the aforementioned conformal results
of \cite{Joyce1991,DabkowskiLock,Apostolov-Rollin}.  However, before detailing these, it is first necessary to discuss a conformal relationship between orbifolds an ALE spaces.  Given an ALE space $(X,g)$, choose a conformal factor $u : X \rightarrow \RR_+$ 
such that $u = O(\rho^{-2})$ as $\rho \rightarrow \infty$. 
The space $(X, u^2 g)$ then compactifies to a
$C^{1,\alpha}$-orbifold.  This is known as a {\textit{conformal compactification}}. In the anti-self-dual case, there moreover exists a $C^{\infty}$-orbifold conformal compactification with positive Yamabe invariant, which we denote by $(\widehat{X}, \widehat{g})$, see
\cite{TV2, ChenLeBrunWeber}. 
(Note that this compactification is quite different than the 
complex analytic compactifications considered below in Section \ref{ca_compactifications}.) 
On the other hand, given a compact Riemannian orbifold $(\widehat{X},\widehat{g})$ with nonnegative Yamabe invariant, and letting $G_p$ denote the Green's function of the conformal Laplacian associated to a point $p$, the space $(\widehat{X}\setminus\{p\},G_p^2\widehat{g})$ is a complete noncompact scalar-flat anti-self-dual ALE orbifold.  A coordinate system at infinity arises from using inverted normal coordinates in the metric $\widehat{g}$ around $p$.  We refer to this as a {\textit{conformal blow-up}}.

\begin{remark}
\label{compactification_groups}
{\em 
If $(X,g)$ is an oriented ALE space with group at infinity $\Gamma \subset {\rm{SO}}(4)$, then the 
conformal compactification $(\widehat{X},\widehat{g})$ has the 
same group action at the orbifold point, provided we {\textit{reverse}} the
orientation,
}
\end{remark}

In \cite{Joyce1991}, Joyce showed that the LeBrun negative mass metric $(\mathcal{O}(-n),g_{LB})$ is the conformal blow-up of $(\CP^2_{(1,1,n)},g_{BK})$.  
Later, the first author and Dabkowski found the explicit conformal equivalence between these in \cite{DabkowskiLock}.  While these metrics are K\"ahler, it is important to note that they are K\"ahler with respect to reverse oriented complex structures. 

The key result we will need is the following, recently proved by Apostolov-Rollin 
in \cite{Apostolov-Rollin}.  On $(\CP^2_{(r,q,p)},g_{BK})$, with singular points and orbifold groups as given in \eqref{WPS_table}, let $G_{z}$ denote the Green's function of the conformal Laplacian associated to any one of these singular points $z$.  Then, the space $(\CP^2_{(r,q,p)}\setminus\{z\},G^2_{z} \cdot g_{BK})$, with reversed orientation, is a scalar-flat K\"ahler ALE orbifold with group at infinity the same as the group of the blown-up orbifold point, and two singular points with orbifold groups that are orientation-reversed conjugate to their compactified counterparts. 

%%%%%%%%%%%%%%%%%%%%%%%%%%%%%%%%%%%%%%%%%%%%%%%%%%%%%%%%
\section{Group actions}
\label{group_actions}

\subsection{Quaternions and $S^3$}
\label{quaternions_S3}
It is our convention to identify $\CC^2$ with the space of quaternions, $\mathbb{H}=\{x_0+x_1\hat{i}+x_2\hat{j}+x_3\hat{k}\}$,
by having $(z_1,z_2)\in\CC^2$ correspond to $z_1+z_2\hat{j}\in\mathbb{H}$.  Identifying $S^3$ with the space of unit quaternions in the obvious way, it is easy to see that it acts on itself isometrically, in an orientation preserving way, by both left and right quaternionic multiplication.  
It is well known that the map $\phi:S^3\times S^3\rightarrow {\rm SO}(4)$ defined by
\begin{align}
\label{phi}
\phi(q_1,q_2)(h)=q_1*h*\bar{q}_2,
\end{align}
for $h\in S^3$, is a double cover of ${\rm SO}(4)$.  Clearly, the kernel of $\phi$ is $\{(1,1),(-1,-1)\}$, hence $S^3\times S^3/(-1,-1)={\rm SO}(4)$.  
The restriction of $\phi$ to the diagonal subgroup acts isometrically on the space of purely imaginary unit quaternions, which is identified with $S^2$, and therefore induces the double cover $\psi:S^3\rightarrow {\rm SO}(3)$ defined by
\begin{align}
\label{Psi_SO(3)}
\psi(q)(h)=q*h*\bar{q},
\end{align}
for $h\in S^2$.
For the remainder of this work, in a slight abuse of notation, we will write the elements $\phi(\alpha,\beta)\in{\rm SO}(4)$ as the pair $[\alpha,\beta]$ where $\alpha,\beta\in S^3$, with the action given by $[\alpha,\beta](h)=\alpha*h*\bar{\beta}$ and composition by $[\alpha_2,\beta_2]\circ[\alpha_1,\beta_1]=[\alpha_2*\alpha_1,\beta_1*\beta_2]$.  

Now, we briefly introduce the Hopf fibration and examine its behavior under left and right quaternionic multiplication.  Writing $S^3=\{(z_1,z_2)\in \CC^2: |z_1|^2+|z_2|^2=1\}$ and $S^2=\CC \cup \{\infty\}$, the Hopf map $\mathcal{H}:S^3\rightarrow S^2$ is given by
\begin{align}
\label{Hopf_map}
\mathcal{H}(z_1,z_2)=z_1/z_2.
\end{align}
Observe, from this, that the typical fiber over $w\in S^2=\CC\cup \{\infty\}$ is the circle
\begin{align}
\frac{e^{i\theta}}{\sqrt{|w|^2+1}}(w,1)=\frac{e^{i\theta}}{\sqrt{|w|^2+1}}(w+\hat{j})\text{$\in S^3$, where $0\leq \theta<2\pi$.}
\end{align} 
Let $h=h_1+h_2\hat{j}\in S^3$.  First, by left multiplication we have
\begin{align}
\label{Hopf_rotation}
(h_1+h_2\hat{j})*\frac{e^{i\theta}}{\sqrt{|w|^2+1}}(w+\hat{j})=\frac{e^{i\theta}h_1}{\sqrt{|w|^2+1}}(w+\hat{j})+\frac{e^{-i\theta}h_2}{\sqrt{|w|^2+1}}(-1+\bar{w}\hat{j}),
\end{align}
so we see that the Hopf fibration is preserved if and only if $h_1$ or $h_2$ is zero.  In the case that $h=h_1$, not only is the Hopf fibration preserved, but clearly the action is just rotation in the Hopf fiber.  When $h=h_2\hat{j}$, observe that the entire circle is mapped to the Hopf fiber over the point $-1/\bar{w}\in S^2$.
Next, by right multiplication we have
\begin{align}
\label{Hopf_SU(2)}
\frac{e^{i\theta}}{\sqrt{|w|^2+1}}(w+\hat{j})*(h_1+h_2\hat{j})=\frac{e^{i\theta}}{\sqrt{|w|^2+1}}(h_1w-\bar{h}_2,h_2w+\bar{h}_1).
\end{align}
This is the Hopf fiber over the point $(h_1w-\bar{h}_2)/(h_2w+\bar{h}_1)\in S^2=\CC\cup \{\infty\}$.  Therefore, we see that right quaternionic multiplication always preserves the Hopf fibration.

\subsection{Finite subgroups of ${\rm U}(2)$ containing no complex reflections}
\label{U(2)_subgroups}
Recall the following short exact sequence of Lie groups
\begin{align}
1\longrightarrow {\rm SU}(2)\overset{i}{\longrightarrow}{\rm U}(2)\overset{det}{\longrightarrow} {\rm U}(1)\longrightarrow 1.
\end{align}
Since this splits, we see that ${\rm U}(2)={\rm U}(1)\ltimes {\rm SU}(2)$; the semidirect product can be seen more explicitly in the double cover
\begin{align}
\label{U(2)_cover}
\phi:S^1\times S^3\rightarrow {\rm U}(2)
\end{align}
obtained by restricting the map $\phi$, from \eqref{phi}, to $S^1\times S^3$ for the particular $S^1={\rm U}(1)$ given by the unit quaternions $\exp(i\theta)\in S^3$ for $0\leq \theta<2\pi$, see \cite{Coxeter, Crowe, DuVal}. 

First, we consider the finite subgroups of ${\rm SU}(2)$.
There is an isomorphism between $S^3$ and ${\rm SU}(2)$ via right quaternionic multiplication as follows.  Let $h_1+h_2\hat{j}\in S^3$, then 
\begin{align}
\label{SU(2)_identify}
(z_1+z_2\hat{j})*(h_1+h_2\hat{j})\longleftrightarrow 
\left(
\begin{matrix}
h_1& -\bar{h}_2\\
h_2 & \bar{h}_1
\end{matrix}
\right)
\left(
\begin{matrix}
z_1\\
z_2
\end{matrix}
\right).
\end{align}
Notice, from \eqref{Hopf_SU(2)}, that ${\rm SU}(2)$ preserves the Hopf fibration.
The finite subgroups of ${\rm SU}(2)$ were determined in \cite{Coxeter_1940}.  They can be classified by the simply laced affine Dynkin of ADE-type by the McKay correspondence \cite{McKay}.  The quotient of $\CC^2$ by a finite subgroup of ${\rm SU}(2)$ yields a {\textit{Du Val singularity}} at the origin whose minimal resolution is a tree of smooth rational curves with intersection matrix equivalent to the Cartan matrix of a Dynkin diagram of ADE-type.  There has been much exposition on the  finite subgroups of ${\rm SU}(2)$, so here we provide a brief summary and refer the reader to \cite{Stekolshchik} for more details. 

The finite subgroups of ${\rm SU}(2)$ are the cyclic and binary polyhedral groups.  For each integer $n\geq 2$, right quaternionic multiplication by $\exp(2\pi i/n)$ generates a cyclic group of order $n$, which we denote by $L(-1,n)$.  This is of ADE-type A$_{n-1}$.
By \eqref{SU(2)_identify}, this generator corresponds to the matrix
\begin{align}
\label{L(-1,n)}
\left(
\begin{matrix}
\exp(2\pi i/n)&0\\
0&\exp(-2\pi i/n)
\end{matrix}
\right)
\left(
\begin{matrix}
z_1\\
z_2
\end{matrix}
\right)
\end{align}
acting on $\CC^2$.
A quaternionic representation of the binary polyhedral groups first arose in Coxeter's classification of the finite subgroups of the multiplicative group of quaternions \cite{Coxeter_1940}.
In considering a natural generalization of Hamilton's formula's for the quaternions group, for all triples of integers $a,b,c$, where $2\leq a\leq b\leq c$, he introduced a group defined by
\begin{align}
\label{binary_presentation}
\langle a,b,c\rangle:=\langle R, S, T |R^a=S^b=T^c=RST\rangle,
\end{align}
and proved that it is finite in exactly the following cases
\begin{align}
\label{a,b,c}
a,b,c=\begin{cases}
2,2,n\phantom{=}\text{for all $n\in\mathbb{Z}^+$}\\
2,3,3\\
2,3,4\\
2,3,5.
\end{cases}
\end{align}
Coxeter then showed that these are in fact the binary polyhedral groups,
and proceeded to find generators for each as a set of unit quaternions acting by right quaternionic multiplication.  However, there was a mistake in his work for the binary icosahedral group which Lam later fixed \cite{Lam}.  In Table \ref{SU(2)_table} we provide a list of these groups along with their orders and ADE-types.  The reader will find a list of generators by setting $m=1$ in the first four cases of Table \ref{U(2)_table}.

\begin{table}[ht]
\label{SU(2)_table}
\caption{}
\begin{center}
\renewcommand\arraystretch{1.4}
\begin{tabular}{lllll}
\hline
$\langle a,b,c\rangle$ & Binary polyhedral group & & Order &  ADE-type\\\hline\hline
$\langle2,2,n\rangle$& Binary dihedral group  &$D^*_{4n}\phantom{=}$& $4n$   & $D_{n+2}$ \\
$\langle2,3,3\rangle$&  Binary tetrahedral group &   $T^*\phantom{=}$&$24$    & $E_6$ \\
$\langle2,3,4\rangle$&  Binary octahedral group  &   $O^*\phantom{=}$& $48$    &$E_7$\\
$\langle2,3,5\rangle$& Binary icosahedral group  &  $I^*\phantom{=}$& $120$    & $E_8$ \\\hline
\end{tabular}
\end{center}
\end{table}

Since $RST$ is in the center of $\langle a,b,c\rangle$, Coxeter considered the group quotient by the cyclic subgroup generated by $RST$, which clearly has the presentation
\begin{align}
\label{regular_presentation}
(a,b,c):=\langle r, s, t |r^a=s^b=t^c=rst=1\rangle.
\end{align}
This presentation, for $a,b,c$ corresponding to the binary polyhedral groups, is that of the polyhedral groups in ${\rm SO}(3)$; the symmetry groups of the regular $n$-gon, tetrahedron, octahedron and icosahedron.  
Coxeter showed that $RST=-Id$ in these cases, hence $(a,b,c)$ here is a $\mathbb{Z}_2$ quotient.  Equivalently that each binary dihedral groups is a double cover of the corresponding polyhedral group.
The polyhedral groups act as orientation preserving isometries on $S^2$, and it is well known the quotient of $S^2$ by one of these groups results in an orbifold with three singular points of orders $a,b,c$ for the corresponding polyhedral group.  This is expressed in Table \ref{polyhedral_group}

\begin{table}[ht]
\label{polyhedral_group}
\caption{}
\begin{center}
\renewcommand\arraystretch{1.4}
\begin{tabular}{ llll}
\hline
$(a,b,c)$ & Polyhedral group && Order\\\hline\hline
$(2,2,n)$ & Dihedral group &$D_{2n}\phantom{=}$ & $2n$\\
$(2,3,3)$ & Tetrahedral group &$T\phantom{=}$ & $12$\\
$(2,3,4)$ & Octahedral group &$O\phantom{=}$ &  $24$\\
$(2,3,5)$ &Icosahedral group &$I\phantom{=}$ &$60$\\\hline
\end{tabular}
\end{center}
\end{table}

One can see more clearly how the binary polyhedral groups cover the polyhedral groups by examining
${\rm SU}(2)$-actions under the Hopf map \eqref{Hopf_map}.  Consider an arbitrary element $\gamma=[\exp(2\pi i/m),h_1+h_2\hat{j}]\in{\rm U}(2)$.  Then
\begin{align}
\begin{split}
\label{Hopf_S2_action}
\mathcal{H}(\gamma(z_1,z_2))&=\mathcal{H}\big(e^{2\pi i/m}(h_1\cdot z_1-\bar{h}_2\cdot z_2),e^{2\pi i/m}(h_2\cdot z_1+\bar{h}_1\cdot z_2)\big)\\
&=\frac{h_1\cdot z_1-\bar{h}_2\cdot z_2}{h_2\cdot z_1+\bar{h}_1\cdot z_2}\in S^2=\CC\cup \{\infty\},
\end{split}
\end{align}
so we see that $\gamma$ descends to a Mobius transformation.  Notice that left multiplication by the quaternion $\exp(2\pi i/m)$ cancels out under the Hopf map, which is intuitively sensible since that corresponds to rotation of the Hopf fiber.  It will later be important to understand this for a general ${\rm U}(2)$-action.

\begin{remark}
{\em
Since $RST=-Id$ for each binary polyhedral group, they can be presented in terms of just the generators $S$ and $T$ as
\begin{align}
\label{two_generators}
\langle S, T |(ST)^2=S^b=T^c\rangle,
\end{align}
where $b,c$ are as in $\langle a,b,c\rangle$.
\em}
\end{remark}

Heuristically, from \eqref{U(2)_cover}, one sees that any finite subgroup of ${\rm U}(2)$ will be generated by some combination of a quaternion of the form $\exp(2\pi i/m)$ acting on the left, a finite set of unit quaternions acting on the right, and products thereof.

A classification of the finite subgroups of ${\rm U}(2)$, up to ${\rm U}(2)$-conjugacy, was given in \cite{DuVal, Coxeter}.  
However, here we are interested in those subgroups contain no complex reflections (those which act freely on $S^3$).  These were classified in \cite{Scott}.  Up to certain conditions, they are the image under $\phi$ of products of the cyclic group $L(1,2m)$ with the binary polyhedral groups, the index--$2$ diagonal subgroup of $\phi(L(1,4m)\times D^*_{4n})$ which we denote by $\mathfrak{I}^2_{m,n}$, the index--$3$ diagonal subgroup of $\phi(L(1,6m)\times T^*)$ which we denote by $\mathfrak{I}^3_m$, and the cyclic groups $L(q,p)$.
In Table~\ref{U(2)_table} we list these groups along with their respective conditions and generators written in the form of $[\alpha,\beta]$ for some unit quaternions $\alpha$, $\beta\in S^3$ to be consistent with the notation in Section \ref{quaternions_S3}.  
From \cite{Coxeter_1940, Lam}, a choice of generators in the case of the image under $\phi$ of the product of $L(1,2m)$ with a binary polyhedral group is clear.  However, a set of generators in terms of elements of $S^1\times S^3$ is not obvious for the subgroups $\mathfrak{I}^2_{m,n}$, $\mathfrak{I}^3_m$ and $L(q,p)$, and we give a proof for these cases below.
\begin{table}[ht]
\caption{Finite $\Gamma\subset{\rm U}(2)$ containing no complex reflections}
\label{U(2)_table}
\renewcommand\arraystretch{1.4}
\begin{center}
\begin{tabular}{llll}
\hline
$\Gamma\subset {\rm U}(2)$ & Order & Generators\\\hline\hline
$\bullet\phantom{i}\phi(L(1,2m)\times D^*_{4n})$  & $4mn$ & $[e^{\pi i/m},1]$, $[1,e^{\pi i/n}]$, $[1,\hat{j}]$\\
\tiny{with $(m,2n)=1$}&&\\
$\bullet\phantom{i}\phi(L(1,2m)\times T^*)$ &  $24m$ &$[e^{\pi i/m},1]$, $[1, (1+\hat{i}+\hat{j}-\hat{k})/2]$, $[1, (1+\hat{i}+\hat{j}+\hat{k})/2]$\\
\tiny{with $(m,6)=1$}&&\\

$\bullet\phantom{i}\phi(L(1,2m)\times O^*)$ &  $48m$ & $[e^{\pi i/m}, 1]$, $[1, e^{\pi i /4}]$, $[1, (1+\hat{i}+\hat{j}+\hat{k})/2]$\\
\tiny{with $(m,6)=1$}& &  \\

$\bullet\phantom{i}\phi(L(1,2m)\times I^*)$ &  $120m$ & $[e^{\pi i/m}, 1]$, $[1, (1+\tau\hat{i}-\tau^{-1}\hat{k})/2]$, $[1, (\tau+\hat{i}+\tau^{-1}\hat{j})/2]$\\
\tiny{with $(m,30)=1$}& &  \tiny{where $\tau=(1+\sqrt{5})/2$}\\

$\bullet\phantom{i} \mathfrak{I}^2_{m,n}$ & $4mn$ &$[e^{\pi i/m},1]$, $[1,e^{\pi i/n}]$, $[e^{\pi i/(2m)},\hat{j}]$\\
\tiny{with $(m,2)=2$, $(m,n)=1$}&&\\

$\bullet\phantom{i} \mathfrak{I}^3_m$ & $24m$ & $[e^{\pi i/m}, 1]$, $[1,\hat{i}]$, $[1,\hat{j}]$, $[e^{\pi i /(3m)},(-1-\hat{i}-\hat{j}+\hat{k})/2 ]$\\
\tiny{with $(m,6)=3$} &&\\

$\bullet\phantom{i} L(q,p)$ & $p$ & $[e^{2\pi i k/p}, e^{2\pi i (1-k)/p}]$\\
\tiny{with $(q,p)=1$} & & \tiny{where $2k\equiv (q+1)$mod $p$}\\\hline

\end{tabular}
\end{center}
\end{table}

\begin{remark}
{\em
Notice that the groups $\phi(L(1,2m)\times D^*_4)$ and $\mathfrak{I}^2_{m,1}$ (the $n=1$ cases) are in fact cyclic, and are therefore excluded when {\textit{non-cyclic}} subgroups are considered.
}
\end{remark}

Generators of $\mathfrak{I}^2_{m,n}$ were found in \cite{FalbelPaupert}, however we find it prudent to discuss them here as well.  Consider the index--$2$ normal subgroups, $L(1,2m)\lhd L(1,4m)$ and $L(-1,2n)\lhd D^*_{4n}$, generated by $[e^{\pi i/m}, 1]$ and $[1, e^{\pi i /n}]$ respectively.  Then $\phi(L(1,2m)\times L(-1,2n))\subset \phi(L(1,4m)\times D^*_{4n})$ is an index--$4$ subgroup, so to obtain the index--$2$ diagonal subgroup one only needs to add the generator $[e^{\pi i /(2m)}, \hat{j}]$.  In fact, $\mathfrak{I}^2_{m,n}$ is generated by $[1, e^{\pi i /n}]$ and $[e^{\pi i /(2m)}, \hat{j}]$ alone since $[e^{\pi i /(2m)}, \hat{j}]^2=[-e^{\pi i /m},1]$ and $-Id\in L(-1,2n)$, however it will become apparent later why it is useful to consider the set of three generators.

Generators of $\mathfrak{I}^3_m$ can be found in a similar way to the index--$2$ diagonal case discussed above.  However, they are not found in  \cite{FalbelPaupert}.
Consider the index--$3$ normal subgroups, $L(1,2m)\lhd L(1,6m)$ and $D^*_{8}\lhd T^*$,
generated by $[e^{\pi i/m}, 1]$ and $\{[1,\hat{i}],[1,\hat{j}]\}$ respectively.  Then $\phi(L(1,2m)\times D^*_8)\subset \phi(L(1,6m)\times T^*)$ is an index--$9$ subgroup.
Now, from Table \ref{U(2)_table}, recall that $T^*$ is generated by
\begin{align}
S=[1, (1+\hat{i}+\hat{j}-\hat{k})/2]\phantom{=}\text{and}\phantom{=}T=[1, (1+\hat{i}+\hat{j}+\hat{k})/2],
\end{align}
which are both of order $6$, and notice that $ST=[1,\hat{j}]$ and $TS=[1,\hat{i}]$.  Therefore, to obtain the index--$3$ diagonal subgroup one only needs to add the diagonal element $[e^{\pi i /(3m)},(-1-\hat{i}-\hat{j}+\hat{k})/2 ]$ as a generator, which is  effectively of order $3$ in the extension of  $\phi(L(1,2m)\times D^*_8)\subset \phi(L(1,6m)\times T^*)$ since $(-1-\hat{i}-\hat{j}+\hat{k})/2=(S^2)^{-1}$.  The choice of $(S^2)^{-1}$ is because, via $\phi$, this generator will multiply by the quaternion $S^2$ on the right.  Note that there are several choices one could make here for the diagonal element to add as a generator, however given the presentations \eqref{binary_presentation} and \eqref{two_generators} it is easy to see that all such choices would yield the same group.  Since $[e^{\pi i /(3m)},(-1-\hat{i}-\hat{j}+\hat{k})/2 ]^3=[\exp(\pi i /m),1]$, the entire group $\mathfrak{I}^3_m$ can be generated without $[e^{\pi i/m},1]$, but once again it will become apparent later why it is useful to include this element in the set of generators.

The group $L(q,p)$, where $1\leq q<p$ are relatively prime, is generated by the matrix
\begin{align}
\label{L(q,p)_generator}
\left(
\begin{matrix}
\exp(2\pi i /p)&0\\
0&\exp(2\pi iq/p)
\end{matrix}
\right)
\left(
\begin{matrix}
z_1\\
z_2
\end{matrix}
\right)
\end{align}
acting on $\CC^2$.
We know that this generator must have an equivalent representation as
\begin{align}
[e^{2\pi i k/p},h_1+h_2\hat{j}]
\end{align}
for some $k\in \mathbb{Z}$ and some $h_1+h_2\hat{j}\in S^3$.  Therefore
\begin{align}
\label{L(q,p)_matrix}
\exp(2\pi i k/p)\left(
\begin{matrix}
h_1& -\bar{h}_2\\
h_2 & \bar{h}_1
\end{matrix}
\right)=
\left(
\begin{matrix}
\exp(2\pi i /p)&0\\
0&\exp(2\pi iq/p)
\end{matrix}
\right),
\end{align}
 so $h_2=0$, $e^{2\pi i k/p}\cdot h_1=e^{2\pi i /p}$, and $e^{2\pi i k/p}\cdot \bar{h}_1=e^{2\pi i q/p}$.  From this, observe that $h_1=e^{2\pi i (1-k)/p}$.  Finally, taking determinants of both sides of \eqref{L(q,p)_matrix} we find that $e^{4\pi i k/p}=e^{2\pi i (q+1)/p}$, so $2k\equiv (q+1)$mod $p$.

%%%%%%%%%%%%%%%%%%%%%%%%%%%%%%%%%%%%%%%%%%
\section{ALE orbifold metrics}
\label{quotients_g_LB}
In this section, for each non-cyclic finite subgroup $\Gamma\subset {\rm U}(2)$ containing no complex reflections,
we take the essential step of obtaining a scalar-flat K\"ahler ALE orbifold with group at infinity $\Gamma$ and with all singularities isolated and of cyclic type.  
These metrics will be obtained as certain quotients of certain LeBrun negative mass metrics, and all singularities will lie on an orbifold quotient of the $\CP^1$ at the origin.
\begin{theorem}
\label{quotient_theorem}
For each non-cyclic finite subgroup $\Gamma\subset{\rm U}(2)$ containing no complex reflections, there exists a scalar-flat K\"ahler ALE orbifold with group at infinity $\Gamma$ and with all singularities isolated and specified precisely as follows.
{\em
\begin{align*}
\begin{array}{lllll}
\hline
\Gamma\subset {\rm U}(2)& \text{Conditions}  &  &\hspace{-3mm}\text{Orbifold}& \hspace{-7mm}\text{groups} \\\hline\hline
\phi(L(1,2m)\times D^*_{4n}) & (m,2n)=1 & L(1,2)& L(1,2)& L(-m,n)\\  
\phi(L(1,2m)\times T^*)  & (m,6)=1& L(1,2) &  L(-m,3) & L(-m,3)\\ 
\phi(L(1,2m)\times O^*) & (m,6)=1 & L(1,2)& L(-m,3) & L(-m,4)\\
\phi(L(1,2m)\times I^*)  & (m,30)=1& L(1,2)& L(-m,3) & L(-m,5) \\
\text{$\mathfrak{I}^2_{m,n}$} &(m,2)=2,(m,n)=1& L(1,2) &L(1,2)&L(-m,n) \\
 \text{$\mathfrak{I}^3_m$}& (m,6)=3 & L(1,2)& L(1,3) &L(2,3)\\\hline
\end{array}
\end{align*}
\em}
\end{theorem}

\begin{proof}
Observe, from Table \ref{U(2)_table}, that all $\Gamma$ of interest contain $[e^{\pi i/m},1]$ in the set of generators.  This acts on $S^3$ as rotation by $\pi/m$ in the Hopf fiber, and generates $L(1,2m)$.  Therefore, we desire to take a quotient of $(\mathcal{O}(-2m),g_{LB})$ by
$\Gamma/L(1,2m)$ as to obtain the action of $\Gamma$ at ALE infinity.  This is equivalent to taking a quotient by $\Gamma'\subset\Gamma$, where $\Gamma'$ is the subgroup generated by the same generators listed in Table \ref{U(2)_table}, but without $[e^{\pi i/m},1]$.  However, notice that this action is not effective as $L(1,2m)\cap \Gamma'=-Id$.

From the potential function construction of the LeBrun negative mass metrics discussed in Section \ref{LB_metrics}, it is clear that the ${\rm U}(2)$-action on $(\mathcal{O}(-2m),g_{LB})$ descends from the usual action on $\CC^2\setminus\{0\}$ away from the $\CP^1$ at the origin.  Intuitively the orbifold points of the quotient $(\mathcal{O}(-2m),g_{LB})/\Gamma'$ can be seen to lie on the $\CP^1$ at the origin, and to arise from subgroups of the action on $S^3/L(1,2m)$ that have some particular Hopf fibers on which they act by rotation.  This is so, because as the fibers shrink, singularities occur precisely where some subgroup preserves the fiber as opposed to mapping between different fibers.  The map $\psi:S^3\rightarrow {\rm SO}(3)$, from \eqref{Psi_SO(3)}, can be used to find the fixed points of the induced action on the $\CP^1$ at the origin.  However, to find the actual orbifold groups, one needs to understand in what way the $L(1,2m)$-action in the normal directions determines the actual type singularity, and this is not immediately clear.  We do so by first, for each $\Gamma'$, finding the number of orbifold points and their respective orders, and then by using a diagonalization argument along with the Hopf map to find the precise orbifold groups.

For $\Gamma=\phi(L(1,2m)\times\langle a,b,c\rangle)$, where here the integers $a,b,c$ are as in Table \ref{SU(2)_table} such that they correspond to a binary polyhedral group, observe that $\Gamma'=\langle a,b,c\rangle\subset{\rm SU}(2)$.
Then, since $L(1,2m)$ acts on $S^3$ by rotation of the Hopf fibers, $\Gamma'$ descends to the action of the corresponding 
polyhedral group $(a,b,c)$ on the $\CP^1$ at the origin.  Recall, from Section \ref{U(2)_subgroups}, that $S^2/(a,b,c)$ has three singularities of orders $a,b,c$.
Since $L(1,2m)\cap\Gamma'=-Id$, from the presentations \eqref{binary_presentation} and \eqref{regular_presentation}, we see that the 
effective action of $\Gamma'$, which is $\Gamma/L(1,2m)\subset {\rm U}(2)/L(1,2m)$, is that of the corresponding polyhedral group.
Therefore, $(\mathcal{O}(-2m),g_{LB})/\langle a,b,c\rangle$ has three cyclic singularities of orders $a,b,c$ on the $\CP^1$ at the origin, and the subgroups of $\Gamma'$ fixing these points are cyclic of order twice that of their respective singularities.  Each such subgroup actually fixes two points on the $\CP^1$, but there are identifications of this larger set of fixed points in the overall action of $\Gamma'$ yielding the three singularities.  Thus,
it is only necessary to understand the action around a representative fixed point of each class to determine the  precise orbifold groups.  
These subgroups, being cyclic in ${\rm SU}(2)$, can be respectively diagonalized to $L(-1,2\mathfrak{p})$, generated by $\gamma(\mathfrak{p})=[1,e^{\pi i/\mathfrak{p}}]$, where $\mathfrak{p}\in\{a,b,c\}$ corresponds to the order of the particular singularity.  Due to the ${\rm U}(2)$-invariance of the LeBrun negative mass metrics, we can restrict our attention to finding the orbifold group in the quotient by a diagonalized subgroup.

To do this, examine the induced action of $\gamma(\mathfrak{p})$ under the Hopf map as in \eqref{Hopf_S2_action}:
\begin{align}
\label{Hopf_diagonal}
\mathcal{H}\big(\gamma(\mathfrak{p})(z_1,z_2)\big)=\frac{e^{2\pi i/p}\cdot z_1}{z_2}\in S^2=\CC \cup \{\infty\}.
\end{align}
This fixes the points $\{0\},\{\infty\}\in S^2$.  Without loss of generality, we focus on $\{0\}$.  From \eqref{Hopf_diagonal} observe that $\gamma(\mathfrak{p})$ induces an action of rotation by $2\pi /\mathfrak{p}$ on the tangent plane to $\{0\}\in S^2$.  Since the fixed point corresponds to $z_1=0$, the normal fiber here is the image of the complex line $(0,z_2)$ in the quotient $\CC^2/L(1,2m)$.  
Then, since $\gamma(\mathfrak{p})(0,z_2)=(0,e^{-\pi i/p}\cdot z_2)$, the induced action in the normal directions to the fixed point in $(\mathcal{O}(-2m),g_{LB})$ is rotation by $-2\pi i m/\mathfrak{p}$.
Therefore, the orbifold groups of the singularities in $(\mathcal{O}(-2m),g_{LB})/\langle a,b,c\rangle$ are $L(-m,a)$, $L(-m,b)$ and $L(-m,c)$.

The proofs of the $\mathfrak{I}^2_{m,n}$ and $\mathfrak{I}^3_m$ diagonal subgroup cases will follow similarly, however a delicate issue arises due to the fact that not all generators of the respective $\Gamma'$ here are in ${\rm SU}(2)$.  Therefore, consider the projection map $\Pi:{\rm U}(2)\rightarrow {\rm SU}(2)$ defined by
\begin{align}
\label{SU(2)_projection}
\Pi\big([e^{i\theta},h_1+h_2\hat{j}]\big)=[1,h_1+h_2\hat{j}]\in {\rm SU}(2).
\end{align}
For all $\gamma\in {\rm U}(2)$, from  \eqref{Hopf_S2_action}, we see 
 that $\mathcal{H}(\gamma(z_1,z_2))=\mathcal{H}(\Pi(\gamma)(z_1,z_2))$,
and therefore, the singularities arising  in $(\mathcal{O}(-2m),g_{LB})/\Gamma'$ here, correspond point-wise, and in order, to those arising in the quotient by $\Pi(\Gamma')\subset{\rm SU}(2)$.  Also, certain elements of each $\Gamma'$ contain left quaternionic multiplication, it is necessary to examine their action carefully to determine the orbifold groups.

The subgroup $\Gamma'\subset \mathfrak{I}^2_{m,n}$  descends to the action of the dihedral group $(2,2,n)$ on $S^2$, and therefore $(\mathcal{O}(-2m),g_{LB})/\Gamma'$ has three singularities of orders $2,2,n$ as above.  Clearly, the order $2$ singularities have orbifold groups $L(1,2)$.
The singularity of order $n$ arises as a fixed point of the subgroup generated by $[1, e^{\pi i/n}]$, so the orbifold group will be $L(-m,n)$ as we saw earlier.

The subgroup $\Gamma'\subset \mathfrak{I}^3_m$  descends to the action of the tetrahedral group $(2,3,3)$ on $S^2$, and therefore $(\mathcal{O}(-2m),g_{LB})/\Gamma'$ has three singularities of order $2,3,3$ as above.  Clearly, the order $2$ singularity has orbifold group $L(1,2)$.
The order $3$ singularities arise as the fixed points of the subgroup generated by
$[e^{\pi i/(3m)},(-1-\hat{i}-\hat{j}+\hat{k})/2]$.  This diagonalizes to $[e^{\pi i/(3m)},e^{\pi i/3}]$, the induced action of which under the Hopf map fixes the points $\{0\},\{\infty\}\in S^2=\CC^2\cup\{\infty\}$.  However, here these points are not identified in the overall quotient by $\Gamma'$.  This can be seen by examining the action of the tetrahedral group on $S^2$.  From \eqref{Hopf_diagonal} and \eqref{SU(2)_projection}, observe that the actions in the tangent directions at $\{0\}$ and $\{\infty\}$ are rotations by $2\pi/3$ and $-2\pi/3$ respectively.  The normal fibers to $\{0\}$ and $\{\infty\}$ are the images of the complex lines $(0,z_2)$ and $(z_1,0)$ in $\CC^2/L(1,2m)$ respectively, and therefore the corresponding actions in the normal directs are rotations by $2\pi i(1-m)/3$ and $2\pi i(1+m)/3$.
Finally, since $(m,6)=3$, the orbifold groups are $L(1-m,3)=L(1,3)$ and $L(-1-m,3)=L(2,3)$ respectively.
\end{proof}

\begin{remark}\label{alphabeta}
{\em
In Theorem \ref{quotient_theorem}, there are always precisely $3$ 
singularities, and we will write these as type 
$L(\alpha_i, \beta_i)$ for $i =1,2,3$, where 
$\alpha_i$ is chosen modulo $\beta_i$ to satisfy 
$1 \leq \alpha_i \leq \beta_i -1$. 
}
\end{remark}

%%%%%%%%%%%%%%%%%%%%%%%%%%%%%%%%%%%%%%%%%%%%%%%%%%%%%%%%%%%%%%%%%%

\section{Scalar-flat K\"ahler ALE metrics on minimal resolutions}
\label{SFKsec}
In this section we construct scalar-flat K\"ahler ALE metrics on the minimal resolution of $\CC^2/\Gamma$ for each non-cyclic finite subgroup $\Gamma\subset{\rm U}(2)$ containing no complex reflections.  This will prove Theorem \ref{t1} as the existence of such metrics for cyclic groups is already known.  
We begin in Section \ref{SFK_gluing} by giving an adaptation of a gluing theorem of Rollin-Singer.  Then, we complete the construction in Section \ref{t1_proof} by using this gluing to resolve the singularities of the ALE orbifolds obtained in Theorem \ref{quotient_theorem} by attaching appropriate Calderbank-Singer spaces, and thus obtain smooth scalar-flat K\"ahler ALE metrics. Finally, we make some remarks about the topology of the minimal resolutions.

\subsection{Scalar-flat K\"ahler gluing}
\label{SFK_gluing}
%%%%%%%%%%%%%%%%%%%%%%%%%%%%%%%%%%%%%%%%%%%%%%%%%%%%%
An essential component of the proof of Theorem~\ref{t1} is the following adaptation of a gluing theorem due to Rollin-Singer:

\begin{theorem}
\label{SFK_gluing_theorem} Let $(M, \omega)$ be a scalar-flat K\"ahler ALE orbifold of
complex dimension~$2$ with finitely many cyclic singularities
satisfying $H^1(M ; \RR) = 0$. 
Then the minimal resolution $\widehat{M}$ admits scalar-flat K\"ahler metrics. 
\end{theorem}

The proof is more or less identical to that of Rollin-Singer, with the 
addition of a weight at infinity, and we will only discuss the few places
in the proof that need modification. 
Conforming to the notation from \cite{RollinSinger}, we let 
$(X_1, g_1)$ be a Calderbank-Singer scalar-flat K\"ahler ALE space, and let 
$(X_2, g_2)$ be the scalar-flat K\"ahler ALE orbifold minus a singular point $p$. 
We have holomorphic coordinates $z$ near infinity on $X_1$, with
\begin{align}
\begin{split}
g_1 &= |dz|^2 + \eta_1 (z), \\
| \nabla^m \eta_1 |_{g_1} &= O ( |z|^{-m-2} ),
\end{split}
\end{align}
as $|z| \rightarrow \infty$,  and 
holomorphic coordinates $u$ around the singular point $p$.  On $X_2$
\begin{align}
\begin{split}
g_2 &= |du|^2 + \eta_2(u),\\ 
|\eta_2(u)| &= O ( |u|^2), \\
|\nabla \eta_2 (u)| &= O(|u|),\\ 
|\nabla^{m+2} \eta_2 (u) | & = O(1),
\end{split}
\end{align}
for $m \geq 0$, as $|u| \rightarrow 0$. 

Finally, choose ALE coordinates $w$ near infinity on $X_2$ with 
\begin{align}
\begin{split}
g_2 &= |dw|^2 + \eta_1 (w), \\
| \nabla^m \eta_1 |_{g_1} &= O ( |w|^{-m-2} ),
\end{split}
\end{align}
for $m \geq 0$, as $|w| \rightarrow \infty$. Note that 
these coordinates are not necessarily holomorphic.

Let $r_1 \geq 1$ be a smooth function satisfying
\begin{align}
r_1 =   
\begin{cases}
|z|  &  |z| \geq 2\\
1  &   |z| \leq 1/2,
\end{cases}
\end{align}
and $r_2 \geq 1/2$ for $|u| \geq 1/2$ a smooth function satisfying
\begin{align}
r_2 =   
\begin{cases}
|u|  &  |u| \leq 1/2\\
1  &   |u| \geq 2,
\end{cases}
\end{align}
and let $r_3 \geq 1$ be a smooth function satisfying
\begin{align}
r_3 =   
\begin{cases}
|w|  &  |w| \geq 2\\
1  &   |w| \leq 1/2.
\end{cases}
\end{align}

The weighted space on $X_1$ is defined exactly as before, 
\begin{align}
r_1^{\delta_1}B^{n,\alpha}(X_1)  = \{ f : r_1^{-\delta_1} f \in B^{n, \alpha} (X_1) \},
\end{align}
with norm
\begin{align}
\Vert f \Vert_{n,\alpha, \delta_1} = \Vert r_1^{-\delta_1} f \Vert_{n, \alpha}
\end{align}
where $C^{n, \alpha}$ are the H\"older spaces with norms defined as 
on \cite[page 257]{RollinSinger}

For $X_2$ we will define a doubly-weighted space by
\begin{align}
r_2^{\delta_2} r_3^{\delta_3} B^{n,\alpha}(X_2)= \{ f : r_2^{-\delta_2} r_3^{-\delta_3} f \in B^{n, \alpha} (X_2)\},
\end{align}
with norm
\begin{align}
\Vert f \Vert_{n,\alpha, \delta, \delta_3} = \Vert r_2^{-\delta_2} r_3^{-\delta_3} f \Vert_{n, \alpha}
\end{align}
with obvious modification to the definition of the H\"older norm. 

Fixing a scalar-flat K\"ahler metric, the nonlinear map is 
\begin{align}
\mathcal{F} : \Lambda^2_- (\cong \Lambda_0^{1,1}) \rightarrow \Lambda^2_+, 
\end{align}
defined by 
\begin{align}
\mathcal{F}(A) = \Pi_{\Lambda^2_+}\Lambda_A W_+ ( \omega + A),
\end{align}
where $W_+$ is the self-dual part of the Weyl tensor, 
and $\Lambda$ is the adjoint of wedging with $\omega + A$. 
The zeroes of $\mathcal{F}$ correspond to anti-self-dual Hermitian metrics.

The linearized operator of $\mathcal{F}$ at $A = 0$ is given by
\begin{align}
S : \alpha \mapsto d^+ \delta \alpha + \langle \rho, \alpha \rangle \omega, 
\end{align}
where $\rho$ is the Ricci form. 
The most important step is to show that there is 
no cokernel for the adjoint of the linearized operator on each of the pieces. 
\begin{proposition} Choose $0 < \delta < 2$ and $0 < \alpha < 1$. 
Then on $X_1$, the linearized operator mapping from 
\begin{align}
S_1 : r_1^{-\delta} B^{2, \alpha}(X_1, \Lambda^2_-) \rightarrow 
r_1^{-\delta -2} B^{0,\alpha} (X_1, \Lambda^2_+)
\end{align}
has a bounded right inverse $G_1$. 

On $X_2$, the linearized operator mapping from  
\begin{align}
S_2 : r_2^{-\delta} r_3^{-\delta} B^{2, \alpha}(X_2, \Lambda^2_-) \rightarrow 
r_2^{-\delta -2} r_3^{-\delta} B^{0,\alpha} (X_2, \Lambda^2_+),
\end{align}
has a bounded right inverse $G_2$. 
\end{proposition}
\begin{proof}
The proof is almost identical to that of \cite[Proposition 4.3.2]{RollinSinger}
with a minor modification. 
Writing $\phi \in \Lambda^2_+$ as $\phi  = f \omega + \alpha$, 
where $\alpha \in \Lambda^{2,0} \oplus \Lambda^{0,2}$, a kernel element
of the adjoint operator satisfies
\begin{align}
S^* \phi = d^- \delta ( f \omega + \alpha ) - f \rho = 0. 
\end{align}
This can be differentiated to yield 
\begin{align}
( \Delta^2  + 2 Ric \cdot \nabla^2) f = 0,
\end{align}
which can then be written as 
\begin{align}
(\overline{\partial} \partial^{\#})^* (\overline{\partial} \partial^{\#}) f = 0,
\end{align}
where 
$\partial^{\#}$ is the $(1,0)$ component of the gradient. 

The main point now is that by choice of weight forces 
$f$ to decay, and an integration-by-parts argument shows that 
\begin{align}
(\overline{\partial} \partial^{\#}) f = 0.
\end{align}
Consequently, $\partial^{\#} f$ is holomorphic, which says that 
$J \nabla f$ is Killing. However, there are no decaying Killing 
fields on an ALE space, so $f \equiv 0$, and this implies that $\alpha \equiv 0$. 
\end{proof}

The rest of the gluing argument (construction of the approximate metric and 
application of the implicit function theorem) is almost exactly the same as in \cite{RollinSinger},
and the details are omitted. The final ingredient is an adaptation of a result of Boyer to prove the spaces are in fact scalar-flat K\"ahler. 
To see this, since $g$ is Hermitian, then the K\"ahler form 
$\omega = g(J \cdot, \cdot)$ satisfies
\begin{align}
d \omega +\beta \wedge \omega = 0,
\end{align}
for a $1$-form $\beta$ called the {\textit{Lee form}}. 
In \cite{Boyer}, Boyer shows that if $(X,g)$ is anti-self-dual, then $d^+ \beta = 0$.
From the choice of weight, 
an integration-by-parts argument then shows that $d \beta =0$. 
Since $H^1(X, \RR) = 0$, we have $\beta = df$, and the 
conformal metric 
$e^{f} \omega$ is then K\"ahler (which is necessarily scalar-flat 
since it is anti-self-dual).

\subsection{Uniform construction of the building blocks}
\label{cyclic_construction}
In this section we construct scalar-flat K\"ahler ALE metrics on the minimal resolutions of cyclic quotient singularities.  While these are the previously known metrics discussed in Sections \ref{LB_metrics}--\ref{M-E-H_metrics}, here we will obtain them inductively using only the canonical Bochner K\"ahler metrics on weighted projective spaces of the form $\CP^2_{(1,q,p)}$, which were introduced in Section~\ref{BK_WPS}.  Then, from our use of these in the proof of Theorem \ref{t1} below in Section \ref{t1_proof}, we will see that scalar-flat K\"ahler ALE metrics on every possible minimal resolution can be obtained using this simpler set of building blocks.

From Section \ref{BK_WPS}, recall that the conformal blow-up of $(\CP^2_{(1,1,n)},g_{BK})$ is precisely $(\mathcal{O}(-n),g_{LB})$, see \cite{Joyce1991,DabkowskiLock}.  Note that as $(\mathcal{O}(-n),g_{LB})$ can be obtained from $(\CP^2_{(1,1,n)},g_{BK})$ via conformal blow-up, the scalar-flat K\"ahler ALE orbifolds of Theorem \ref{quotient_theorem} can be seen to arise from these as well.  Also, recall, more generally, that the conformal blow-up of $(\CP^2_{(r,q,p)},g_{BK})$ at any of the singular points yields a scalar-flat K\"ahler ALE orbifold with group at infinity the same as the orbifold group of the point at which the conformal blow-up was based and up to two singular points with orbifold groups orientation reversed conjugate to their compactified counterparts \cite{Apostolov-Rollin}.  

In particular, let us examine the conformal blow up of $(\CP^2_{(1,q,p)},g_{BK})$ at the point $[0,0,1]$.  This will be a scalar-flat K\"ahler ALE orbifold with group at infinity $L(q,p)$ and one singular point on a $\CP^1$ with orbifold group $L(q-p^{-1;q},q)$, recall \eqref{WPS_table} and that $p^{-1;q}$ denotes the inverse of $p$ modulo $q$.  
From these spaces, we will inductively construct construct scalar-flat K\"ahler ALE metrics on the minimal resolution of all cyclic quotient singularities.  By way of induction, assume that $q>1$, since the LeBrun negative mass metrics have already been obtained as conformal blow-ups, and that a scalar-flat K\"ahler ALE metric has already been obtained on the minimal resolution of $\CC^2/L(a,b)$ for all relatively prime integers $1\leq a<b<p$.  Then, to obtain a scalar-flat K\"ahler ALE metric on the minimal resolution of $\CC^2/L(q,p)$ take the conformal-blow up of $(\CP^2_{(1,q,p)},g_{BK})$ at the point $[0,0,1]$, and resolve the singularity by attaching a scalar-flat K\"ahler ALE metric on the minimal resolution of $\CC^2/L(q-p^{-1;q},q)$, which we are assumed to have since $q<p$, using the the scalar-flat K\"ahler gluing of Theorem \ref{SFK_gluing_theorem}.

While this completes the proof, for the sake of clarity we write out the first few steps of this iterative process.  For $p=2$, we have $(\mathcal{O}(-2),g_{LB})$ as the conformal blow-up of $(\CP^2_{(1,1,2)},g_{BK})$.  For $p=3$, we have $(\mathcal{O}(-3),g_{LB})$ as the conformal blow-up of $(\CP^2_{(1,1,3)},g_{BK})$.  To obtain a scalar-flat K\"ahler ALE metric on the minimal resolution of $\CC^2/L(2,3)$, take the conformal blow up of $(\CP^2_{(1,2,3)},g_{BK})$ at the point $[0,0,1]$ to obtain a scalar-flat ALE orbifold with group at infinity $L(2,3)$ and one singular point with orbifold group $L(1,2)$.  Then, since we already have a scalar-flat K\"ahler ALE metric
on the  minimal resolution of $\CC^2/L(1,2)$ by an earlier step in this process, we resolve this singularity by attaching it using Theorem \ref{SFK_gluing_theorem}, and obtain the desired space.

\begin{remark}
{\em
Notice that only weighted projective spaces with one or two singularities, i.e. those of the form $\CP^2_{(1,q,p)}$ where $1\leq q<p$ are relatively prime integers, are made use of in this construction.  This is so because taking the Green's function metrics
of a weighted projective space with three singular points, i.e. one of the form $\CP^2_{(r,q,p)}$ where $1<r<q<p$ are relatively prime integers, and resolving by attaching the appropriate scalar-flat K\"ahler 
ALE spaces to the remaining singularities tends to yield a scalar-flat K\"ahler ALE metric on a blow-up (in the complex analytic sense) of the minimal resolution. 
}
\end{remark}

\subsection{Construction of the metrics for non-cyclic groups}
\label{t1_proof}
For each non-cyclic finite subgroup $\Gamma\subset {\rm U}(2)$ containing no complex reflections, 
consider the scalar-flat K\"ahler ALE orbifold with group at infinity $\Gamma$ 
obtained in
Theorem~\ref{quotient_theorem}.  All singularities of this space lie on the $\CP^1$ 
at the origin and are of cyclic type with the orbifold groups
$L(\alpha_i,\beta_i)$ for $i=1,2,3$, recall Remark \ref{alphabeta}.  
To resolve these singularities, we use 
Theorem \ref{SFK_gluing_theorem} to attach a Calderbank-Singer space with group 
at infinity $L(\alpha_i,\beta_i)$ to the singularity with the corresponding orbifold 
group.  Furthermore, since both the ALE orbifold, obtained as a quotient of a LeBrun 
negative mass metric, and all of the Calderbank-Singer spaces admit holomorphic 
isometric circle actions corresponding to rotations of the Hopf fiber, we can 
impose an $S^1$ symmetry in the gluing argument, that is, we may perform the 
gluing $S^1$-equivariantly, to obtain a scalar-flat K\"ahler ALE space 
with group at infinity $\Gamma$ which admits a holomorphic isometric circle action.

It is straightforward to see that the conditions in Definition \ref{minresdef}
are satisfied, so we have obtained scalar-flat K\"ahler ALE metrics on the minimal resolution of $\CC^2/\Gamma$. This completes the proof of Theorem \ref{t1}

\begin{remark}{\em It would perhaps be possible to use gluing results for constant 
scalar curvature K\"ahler metrics, instead of the Hermitian anti-self-dual 
gluing theorem. We used the latter here since it is easier to generalize the result of
Rollin-Singer to the case that the base orbifold is ALE. However, gluing results for extremal K\"ahler metrics will indeed 
be used in Section \ref{ca_compactifications}. 
}
\end{remark}
\subsection{Topology of the minimal resolutions}
The intersection matrix of the exceptional curve in the minimal 
resolution is described in Figure \ref{resfig}.
\begin{figure}[h]
\setlength{\unitlength}{2cm}
\begin{picture}(7,3.1)(.5,0)
\linethickness{.3mm}
\put(1,0){\line(1,0){1}}
\put(1,0){\circle*{.1}}
\put(.75,-.25){$-e_{k_1}^1$}
\put(2,0){\circle*{.1}}
\put(1.75, -.25){$-e_{k_1-1}^1$}
\multiput(2,0)(0.1,0){10}
{\line(1,0){0.05}}
\put(3,0){\circle*{.1}}
\put(2.75, -.25){$-e_{1}^1$}
\put(3,0){\line(1,0){1}}
\put(4,0){\circle*{.1}}
\put(3.75, -.25){$-b_{\Gamma}$}
\put(4.75, -.25){$-e_1^2$}
\put(4,0){\line(1,0){1}}
\put(5,0){\circle*{.1}}
\put(5,0){\line(1,0){1}}
\put(6,0){\circle*{.1}}
\put(5.75,-.25){$-e_2^2$}
\put(7,0){\circle*{.1}}
\put(6.75, -.25){$-e_{k_2}^2$}
\multiput(6,0)(0.1,0){10}
{\line(1,0){0.05}}
\put(4,0){\line(0,1){1}}
\put(4,1){\circle*{.1}}
\put(3.55, .75){$-e^3_1$}
\put(4,1){\line(0,1){1}}
\put(4,2){\circle*{.1}}
\put(3.55, 1.75){$-e^3_2$}
\multiput(4,2)(0,0.1){10}
{\line(0,1){0.05}}
\put(4,3){\circle*{.1}}
\put(3.55, 2.75){$-e^3_{k_3}$}
\end{picture}
\vspace{5mm}
\caption{The exceptional divisor of the minimal resolution.}
\label{resfig}
\end{figure}
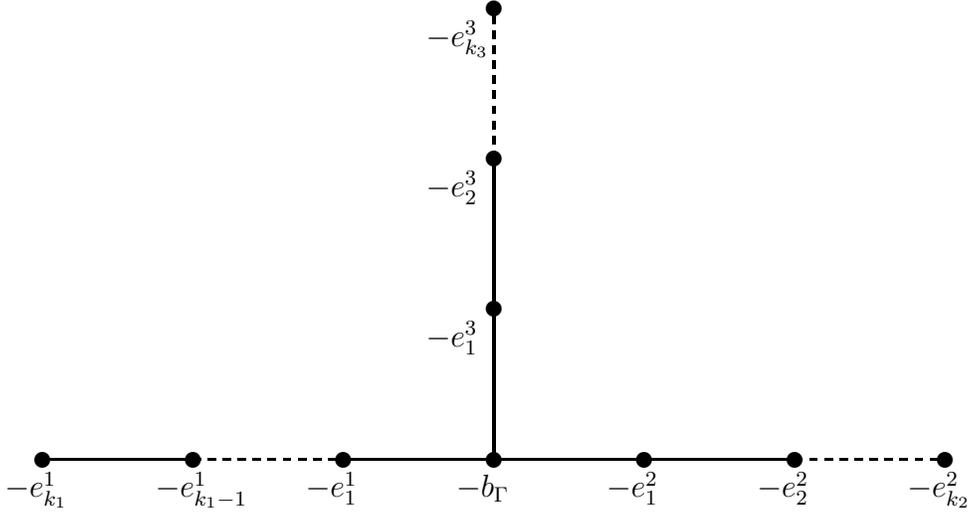

In Figure \ref{resfig}, each string of $e^j_i$ corresponds to the intersection matrix of the respective Calderbank-Singer space attached to the central rational curve at a particular singularity.  Clearly, the signature of the minimal resolution of $\CC^2/\Gamma$ is
\begin{align}
\label{minimal_resolution_signature}
\tau=-b_2^-=-1-\sum_{i=1}^3k_{L(\alpha_i,\beta_i)}.
\end{align}
The orbifold points which arise on the $\CP^1$ at the origin in our work correspond, in fact, to those in Brieskorn's resolution, and our procedure of resolving singularities by attaching the appropriate Calderbank-Singer spaces is a metric version of Brieskorn's complex analytic resolution of those singularities 
by a Hirzebruch-Jung string.

Finally, the self-intersection number of the central rational curve is given by 
\begin{align}
\label{b_self-intersection}
-b_{\Gamma}=-\Big( \frac{\alpha_1}{\beta_1}+\frac{\alpha_2}{\beta_2}+\frac{\alpha_3}{\beta_3}+\frac{2m}{h}\Big),
\end{align}
where $h$ is the order of the image of the group in ${\rm{PGL}}(2,\CC)$ and $m$ corresponds to that of the group, see \cite{Brieskorn}.  When $\Gamma$ is the image under $\phi$ of the product of $L(1,2m)$ with a binary polyhedral group, its image in ${\rm{PGL}}(2,\CC)$ is just the corresponding polyhedral group.  Similarly, the images of $\mathfrak{I}^2_{m,n}$ and $\mathfrak{I}^3_m$ in ${\rm PGL}(2,\CC)$ are the dihedral and tetrahedral groups respectively. Finally, we note that  
\begin{align}
\label{b_Gamma}
b_{\Gamma}=2+\frac{4m}{|\Gamma|}\Big[m-  \Big( m \text{ mod } \frac{|\Gamma|}{4m}\Big)\Big].
\end{align}
Observe, from this, that $-b_{\Gamma}<-1$.  

\section{Scalar-flat K\"ahler deformations}
\label{sfk_deformation}
We begin by finding the dimension of the space of infinitesimal scalar-flat K\"ahler deformations of the ALE orbifolds given in Theorem \ref{quotient_theorem}.  An essential ingredient here is Honda's work in understanding the deformation theory of compactifications of the LeBrun negative mass metrics \cite{HondaOn, Honda_2014}.  The space $(\mathcal{O}(-2),g_{LB})$ is the Eguchi-Hanson metric, for which it is well known that there are no such deformations.  This is the $m=1$ case, and the groups corresponding to this are exactly the binary polyhedral groups.  Therefore it is only necessary to consider $m>1$.  In the following proposition we not only find the dimension of the space of such deformations, but in fact show that it can be given very simply in terms of the respective $b_{\Gamma}$, which recall is negative the self-intersection number of the central rational curve.

\begin{proposition}
\label{sfk_deformation_proposition}
Let $\Gamma\subset {\rm U}(2)$ be a non-cyclic finite subgroup containing no complex reflections with $m>1$.  Then, the dimension of the space of infinitesimal scalar-flat K\"ahler deformations of the quotient $(\mathcal{O}(-2m),g_{LB})/\Gamma'$, which preserve all orbifold singularities, is given by
\begin{align*}
\dim(H^{sfk}_{\Gamma'})=
\displaystyle 2b_{\Gamma}-2=2+\frac{8m}{|\Gamma|}\Big[m-
  \Big( m \text{ \em mod } \frac{|\Gamma|}{4m}\Big)\Big].
\end{align*}
\end{proposition}

\begin{proof}
Finding the dimension of scalar-flat K\"ahler deformation of $(\mathcal{O}(-2m),g_{LB})/\Gamma'$ is equivalent to finding the dimension of the space of scalar-flat K\"ahler deformations of $(\mathcal{O}(-2m),g_{LB})$ that are invariant under the action of $\Gamma'$.
In general, for a representation $\rho$ of a finite group $\Gamma'$, the dimension of the space of invariant elements is given by
\begin{align}
\label{invariant}
\frac{1}{|\Gamma'|}\sum_{\gamma\in \Gamma'}\chi_{\rho}(\gamma),
\end{align}
where $\chi_{\rho}(\gamma)$ denotes the character of $\gamma\in \Gamma'$.  For $\gamma=[1,e^{i\theta}]\neq \pm 1\in {\rm SU}(2)$, there is the following well-known character identity on $S^{2k}(\CC^2)$ as a representation of ${\rm SU}(2)$
\begin{align}
\label{character_identity}
\chi_{S^{2k}(\CC^2)}(\gamma)=\sum_{p=0}^{2k}(e^{i\theta})^{2k-2p}=\frac{\sin((2k+1)\theta)}{\sin(\theta)}=\sin(2k\theta)\cot(\theta)+\cos(2k\theta).
\end{align}
Notice that $\chi_{S^{2k}}(\CC^2)(\gamma=[1,i])=\cos(\pi k)$, since $\sin(\pi k)\cot(\pi/2)=0$.
It will be useful here to introduce the sawtooth function which, for any $x\in \RR$, is defined as
\begin{align}
((x))=\begin{cases}
x-\lfloor x\rfloor -\frac{1}{2}\phantom{=}&\text{when $x\not\in\mathbb{Z}$}\\
0\phantom{=}&\text{when $x\in\mathbb{Z}$},
\end{cases}
\end{align}
where $\lfloor x \rfloor$ denotes the greatest integer less than $x$.
When $\theta=\pi/n$, there is the following useful identity due to Eisenstein, see \cite{Apostol}.
\begin{align}
\label{Eisenstein_identity}
\sum_{j=1}^{n-1}\sin\Big(\frac{2\pi k}{n}j\Big)\cot\Big(\frac{\pi}{n}j\Big)=-2n\Big(\Big(\frac{k}{n}\Big)\Big)
\end{align}

For $m>1$, Honda showed that the complexification of the space of infinitesimal scalar-flat K\"ahler deformations of $(\mathcal{O}(-2m),g_{LB})$ is equivalent to
\begin{align}
\label{representation_SFKd}
\rho\oplus \overline{\rho}\phantom{=}\text{ where }\phantom{=}\rho=S^{2m-2}(\CC^2)\otimes \det,
\end{align}
as a representation space of ${\rm U}(2)$, see \cite{HondaOn}.  Using this along with \eqref{invariant}, we will compute $\dim(H^{sfk}_{\Gamma'})$ for all appropriate $\Gamma'$.

Write the maximal torus in ${\rm U}(2)$ as diag$\{z_1=e^{i\theta_1},z_2=e^{i\theta_2}\}$, and observe that
\begin{align}
\label{character}
\chi_{\rho}(z_1,z_2)=(z_1z_2)\sum_{p=0}^{2m-2}z_1^{2m-2-p}z_2^p,
\end{align}
where $\chi_{\rho}(z_1,z_2)$ denotes the character of any $\gamma\in\Gamma'$ with eigenvalues $\{z_1,z_2\}$.
Note that $\dim(H^{sfk})=2\chi_{\rho}(1,1)=2\chi_{\rho}(-1,-1)=4m-2$.  

Consider when $\Gamma$ is the image under $\phi$ of the product of $L(1,2m)$ and a binary polyhedral group.
(Although the proofs for $\mathfrak{I}^2_{m,n}$ and $\mathfrak{I}^3_m$ will follow similarly, there is a certain delicate difference.)
Here, since the $\Gamma'\subset{\rm SU}(2)$ are the respective binary polyhedral groups so
any $\gamma\in \Gamma'$ has eigenvalues $\{z=e^{i\theta_{\gamma}},z^{-1}=e^{-i\theta_{\gamma}}\}$.  Thus, from \eqref{character_identity}, the character for $\gamma\neq \pm 1$ reduces to
\begin{align}
\begin{split}
\chi_{\rho}(\gamma)&=\sum_{p=0}^{2m-2}(e^{i\theta_{\gamma}})^{2m-2-2p}=\sin((2m-2)\theta_{\gamma})\cot(\theta_{\gamma})+\cos((2m-2)\theta_{\gamma}),
\end{split}
\end{align}
and therefore
\begin{equation}
\label{invariant_elements_SU(2)}
\dim(H^{sfk}_{\Gamma'})=\frac{2}{|\Gamma'|}\Big[2\cdot(2m-1)+\sum_{\gamma\neq \pm Id\in \Gamma'}\sin((2m-2)\theta_{\gamma})\cot(\theta_{\gamma})+\cos((2m-2)\theta_{\gamma})\Big].
\end{equation}
We will now find $\dim(H^{sfk}_{\Gamma'})$ in terms of certain greatest integer functions for each binary polyhedral group separately.  The conjugacy classes of the binary polyhedral groups are well known, see \cite{Stekolshchik}, and from the trace of a representative element one can determine the eigenvalues of the elements in its class.  We use this to decompose the binary polyhedral groups into sets of elements having the same eigenvalues.  Then, we further group the elements in such a way as to be able to repeatedly use \eqref{Eisenstein_identity} to compute $\dim(H^{sfk}_{\Gamma'})$ in terms of sawtooth functions, which we examine with respect to the possible congruencies of $m$ and in turn write in terms of greatest integer functions.

{$\boldsymbol D^{\ast}_{4n}:$} The binary dihedral group is composed of two disjoint sets of elements; the elements of the cyclic group $L(-1,2n)$, and the set of elements
\begin{align*}
S=\{[1,e^{\pi i k/n}*\hat{j}]\}_{1\leq k\leq 2n}.
\end{align*}
Since each element of $S$ has eigenvalues $\{i,-i\}$, and because $((x/2))=0$ for all $x\in \mathbb{Z}$, observe that it does not contribute to the sum in \eqref{invariant_elements_SU(2)}.  Decomposing the sum over these sets, we find that
\begin{equation*}
\begin{split}
\dim(H^{sfk}_{D^*_{4n}})&=\frac{1}{n}\Big[(2m-1)-2n\Big(\Big(\frac{m-1}{n}\Big)\Big)
\\
&\phantom{===} +n\cos(\pi(m-1))+\sum_{j=1}^{n-1}\cos\Big(\frac{2\pi}{n}(m-1)j\Big)\Big]\\
&=\begin{cases}
2\lfloor\frac{m-1}{n}\rfloor+2\phantom{=}& m\not\equiv 1\text{ mod }n\\
2\big(\frac{m-1}{n}\big)+2\phantom{=}&m\equiv 1\text{ mod }n.
\end{cases} 
\end{split}
\end{equation*}

{$\boldsymbol T^*:$} The elements of the binary tetrahedral group are separated into sets according to their eigenvalues as in Table \ref{T*_eigenvalues}.
\begin{table}[ht]
\caption{}
\label{T*_eigenvalues}
\begin{center}
\begin{tabular}{ l | c}
\hline
$\{z_1,z_2\}$ & $\#\{\gamma\in T^*: \text{diag}(\gamma)=\{z_1,z_2\}\}$ \\\hline
$\{1,1\}$ & $1$\\
$\{-1,-1\}$ & $1$\\
$\{i,-i\}$ & $6$\\
$\{e^{\pi i/3},e^{-\pi i/3}\}$ & $8$\\
$\{e^{2\pi i/3},e^{-2\pi i/3}\}$ & $8$
\end{tabular}
\end{center}
\end{table}
Group together the elements with eigenvalues $\{e^{\pi i\ell/3},e^{-\pi i\ell/3}\}$, for $\ell=1,2$, since these constitute the elements of $(L(-1,6)/\mathbb{Z}_2)\setminus{Id})$, and ignore the terms having eigenvalues $\{i,-i\}$ since they do not contribute to the sum in \eqref{invariant_elements_SU(2)}.  Decomposing the sum over these sets, we find that
\begin{equation*}
\begin{split}
\dim(H^{sfk}_{T^*})&=\frac{1}{6}\Big[(2m-1)-24\Big(\Big(\frac{m-1}{3}\Big)\Big)\\
& \phantom{===} +3\cos(\pi(m-1))+4\sum_{j=1}^{2}\cos\Big(\frac{2\pi}{3}(m-1)j\Big)\Big]\\
&=\begin{cases}
4\big\lfloor\frac{m-1}{3}\big\rfloor-m+3\phantom{=}&m\equiv 5\text{ mod }6\\
\frac{m-1}{3}+2\phantom{=}&m\equiv 1\text{ mod }6.
\end{cases}
\end{split}
\end{equation*}

{$\boldsymbol O^*:$} The elements of the binary octahedral group are separated into sets according to their eigenvalues as in Table \ref{O*_eigenvalues}.
\begin{table}[ht]
\caption{}
\label{O*_eigenvalues}
\begin{center}
\begin{tabular}{ l | c}
\hline
$\{z_1,z_2\}$ & $\#\{\gamma\in O^*: \text{diag}(\gamma)=\{z_1,z_2\}\}$\\\hline
$\{1,1\}$ & $1$\\
$\{-1,-1\}$ & $1$\\
$\{i,-i\}$ & $18$\\
$\{e^{\pi i/3},e^{-\pi i/3}\}$ & $8$\\
$\{e^{2\pi i/3},e^{-2\pi i/3}\}$ & $8$\\
$\{e^{\pi i/4},e^{-\pi i/4}\}$ & $6$\\
$\{e^{3\pi i/4},e^{-3\pi i/4}\}$ & $6$
\end{tabular}
\end{center}
\end{table}
Group together the elements with eigenvalues $\{e^{\pi i\ell/3},e^{-\pi i\ell/3}\}$, for $\ell=1,2$, as before.  Also, group together the elements with eigenvalues $\{e^{\pi i\ell/4},e^{-\pi i\ell/4}\}$, for $\ell=1,2,3$, since these constitute $(L(-1,8)/\mathbb{Z}_2)\setminus{Id})$.  Since, the $\{i,-i\}$ elements do not contribute to the sum
in \eqref{invariant_elements_SU(2)}, it does not matter that the order of that set differs from the orders of the other sets in this grouping.  Decomposing the sum over these sets, we find that
 \begin{equation*}
 \begin{split}
 \dim(H^{sfk}_{O^*})&=\frac{1}{12}\Big[(2m-1)-24\Big(\Big(\frac{m-1}{3}\Big)\Big)-24\Big(\Big(\frac{m-1}{4}\Big)\Big)+6\cos(\pi(m-1))\\
&\phantom{====}+4\sum_{j=1}^{2}\cos\Big(\frac{2\pi}{3}(m-1)j\Big)+3\sum_{j=1}^{3}\cos\Big(\frac{2\pi}{4}(m-1)j\Big)\Big]\\
&=\begin{cases}
2\lfloor\frac{m-1}{3}\rfloor+2\lfloor\frac{m-1}{4}\rfloor-m+3\phantom{=}&m\equiv 11\text{ mod }12\\
2\lfloor\frac{m-1}{4}\rfloor+\frac{1-m}{3}+2\phantom{=}&m\equiv 7\text{ mod }12\\
2\lfloor\frac{m-1}{3}\rfloor+\frac{1-m}{2}+2\phantom{=}&m\equiv 5\text{ mod }12\\
\frac{m-1}{6}+2\phantom{=}&m\equiv 1\text{ mod }12.
\end{cases}
\end{split}
\end{equation*}

{$\boldsymbol I^*:$} The elements of the binary icosahedral group are separated into sets corresponding to their eigenvalues as in Table \ref{I*_eigenvalues}.
\begin{table}[ht]
\caption{}
\label{I*_eigenvalues}
\begin{center}
\begin{tabular}{  l | c }
\hline
$\{z_1,z_2\}$ & $\#\{\gamma\in I^*: \text{diag}(\gamma)=\{z_1,z_2\}\}$\\\hline
$\{1,1\}$ & $1$\\
$\{-1,-1\}$ & $1$\\
$\{i,-i\}$ & $30$\\
$\{e^{\pi i/3},e^{-\pi i/3}\}$ & $20$\\
$\{e^{2\pi i/3},e^{-2\pi i/3}\}$ & $20$\\
$\{e^{\pi i/5},e^{-\pi i/5}\}$ & $12$\\
$\{e^{2\pi i/5},e^{-2\pi i/5}\}$ & $12$\\
$\{e^{3\pi i/5},e^{-3\pi i/5}\}$ & $12$\\
$\{e^{4\pi i/5},e^{-4\pi i/5}\}$ & $12$
\end{tabular}
\end{center}
\end{table}
Group together elements with eigenvalues $\{e^{\pi i\ell/3},e^{-\pi i\ell/3}\}$, for $\ell=1,2$, as before.  Also, group together elements with eigenvalues $\{e^{\pi i\ell/5},e^{-\pi i\ell/5}\}$, for $\ell=1,2,3,4$, since these constitute the elements of $(L(-1,10)/\mathbb{Z}_2)\setminus{Id})$.  Ignoring elements with eigenvalues $\{i,-i\}$ since they do not contribute to the sum in \eqref{invariant_elements_SU(2)}, and decomposing the sum over these sets, we find that
\begin{equation*}
\begin{split}
\dim(H^{sfk}_{I^*})&=\frac{1}{30}\Big[(2m-1)-60\Big(\Big(\frac{m-1}{3}\Big)\Big)-60\Big(\Big(\frac{m-1}{5}\Big)\Big)+15\cos(\pi(m-1))\\
&\phantom{====}+10\sum_{j=1}^{2}\cos\Big(\frac{2\pi}{3}(m-1)j\Big)+6\sum_{j=1}^{4}\cos\Big(\frac{2\pi}{5}(m-1)j\Big)\Big]\\
&=\begin{cases}
2\lfloor\frac{m-1}{3}\rfloor+2\lfloor\frac{m-1}{5}\rfloor-m+3\phantom{=}&m\equiv 17, 23, 29\text{ mod }30\\
2\lfloor\frac{m-1}{5}\rfloor+\frac{1-m}{3}+2\phantom{=}&m\equiv 7, 13, 19\text{ mod }30\\
2\lfloor\frac{m-1}{3}\rfloor+3\big(\frac{1-m}{5}\big)+2\phantom{=}&m\equiv 11\text{ mod }30\\
\frac{m-1}{15}+2\phantom{=}&m\equiv 1\text{ mod }30.
\end{cases}
\end{split}
\end{equation*}

The proof for these groups is completed as follows.
We have found each $\dim(H^{sfk}_{\Gamma'})$ in 
terms of certain greatest integer functions with respect to the particular congruences mod $|\Gamma|/4$.   For any positive integers $x,y,z$, satisfying $y<x\text{ mod }z$, notice that 
\begin{align}
\label{greatest_integer_simplify}
\Big\lfloor \frac{x-y}{z}\Big\rfloor=\frac{x- (x\text{ mod }z)}{z}\in\mathbb{Z}.
\end{align}
Using this, the expressions for $\dim(H^{sfk}_{\Gamma'})$ are easily 
simplified to obtain $2b_{\Gamma}-2$.

The idea for the groups $\mathfrak{I}^2_{m,n}$ and $\mathfrak{I}^3_m$ is similar, however as the authors are unaware of a decomposition of these groups into conjugacy classes, we perform an element-wise decomposition of these groups into sets, from which we are able to find the respective eigenvalue decompositions and proceed as above.  We will see that all eigenvalues here will be of the form $\{z_1,z_2\}=e^{i\theta_1}\{e^{i\theta_2},e^{-i\theta_2}\}$, for some $0\leq \theta_1,\theta_2<2\pi$, and therefore prove the following useful identity before we begin.
\begin{align}
\begin{split}
\label{index_identity}
\chi_{\rho}(e^{i(\theta_1+\theta_2)}&,e^{(\theta_1-i\theta_2)})=e^{i(2\theta_1)}\sum_{p=0}^{2m-2}\Big(e^{i(\theta_1+\theta_2)}\Big)^{2m-2-p}\Big(e^{i(\theta_1-\theta_2)}\Big)^p\\
&=e^{i(2m\theta_1)}\sum_{p=0}^{2m-2}\Big(e^{i\theta_2}\Big)^{2m-2-2p}\\
&=e^{i(2m\theta_1)}\big[\sin(2(m-1)\theta_2)\cot(\theta_2)+\cos(2(m-1)\theta_2)\big].
\end{split}
\end{align}

{$ \boldsymbol{\mathfrak{I}^2_{m,n}:}$} Recall that $\Gamma'\subset {\rm U}(2)$ is generated by $[1,e^{\pi i/n}]$, $[e^{\pi i/(2m)},\hat{j}]$.  In Section~\ref{U(2)_subgroups},  we saw that these actually generate the entire group $\mathfrak{I}^2_{m,n}$, but in order to obtain a smooth quotient it was necessary to take a quotient of $(\mathcal{O}(-2m),g_{LB})$, and included $[e^{\pi i/m},1]$ in the set of generators accordingly.  To find $\dim(H^1_{\mathfrak{I}^2_{m,n}})$, we first decompose $\Gamma'=\mathfrak{I}^2_{m,n}$ into the two disjoint sets of elements
\begin{itemize}
\item $S_1=\Big\{\gamma_1^{\ell,k}:=[e^{\pi i/(2m)},\hat{j}]^{2\ell}\cdot[1,e^{\pi ik/n}]\Big\}_{0\leq \ell\leq m-1 \text{ and }0\leq k\leq 2n-1}$\\
\item $S_2=\Big\{\gamma_2^{\ell,k}:=[e^{\pi i/(2m)},\hat{j}]^{2\ell+1}\cdot[1,e^{\pi ik/n}]\Big\}_{0\leq \ell\leq m-1 \text{ and }0\leq k\leq 2n-1}.$
\end{itemize}
Clearly $S_1,S_2\subset \mathfrak{I}^2_{m,n}$ and $S_1\cap S_2=\varnothing$, so since $|S_1|=|S_2|=2mn$ and $|\mathfrak{I}^2_{m,n}|=4mn$ we see that  $S_1\sqcup S_2$ constitutes all of $\mathfrak{I}^2_{m,n}$.  Next, for arbitrary $\gamma_1^{\ell,k}\in S_1$ and $\gamma_2^{\ell,k}\in S_2$, as given above, observe that
\begin{itemize}
\item 
The eigenvalues of $\gamma_1^{\ell,k}\in S_1$ are 
$(-1)^\ell\{e^{\pi i (\frac{\ell}{m}+\frac{k}{n})},e^{\pi i (\frac{\ell}{m}-\frac{k}{n})}\}$
\item
The eigenvalues of $\gamma_2^{\ell,k}\in S_2$ are
$(-1)^\ell\{e^{\pi i (\frac{2\ell+1}{2m}+\frac{1}{2})},e^{\pi i (\frac{2\ell+1}{2m}-\frac{1}{2})}\}$.
\end{itemize}
From \eqref{index_identity}, the characters of these elements are found to be as follows.
\begin{itemize}
\item $\chi_{\rho}(\gamma_1^{\ell,k})=\begin{cases}
\sin\big(\frac{2\pi}{n}k(m-1)\big)\cot\big(\frac{\pi}{n}k\big)+\cos\big(\frac{2\pi}{n}k(m-1)\big)\phantom{=}&\text{$k\neq 0$ or $n$}\\
2m-1\phantom{=}&\text{$k=0$ or $n$}.
\end{cases}$
\item $\chi_{\rho}(\gamma_2^{\ell,k})=1$ since $m$ is even.
\end{itemize}
Therefore, decomposing the sum over these sets, we find that
\begin{equation*}
\begin{split}
\dim(H^{sfk}_{\mathfrak{I}^2_{m,n}})&=\frac{1}{2mn}\Big[2m\cdot(2m-1)-4mn\Big(\Big(\frac{m-1}{n}\Big)\Big)
\\
&\phantom{=====}+2mn+2m\sum_{j=1}^{n-1}\cos\Big(\frac{2\pi}{n}(m-1)j\Big)\Big]\\
&=\begin{cases}
2\lfloor\frac{m-1}{n}\rfloor+2\phantom{=}& m\not\equiv 1\text{ mod }n\\
2\big(\frac{m-1}{n}\big)+2\phantom{=}&m\equiv 1\text{ mod }n.
\end{cases} 
\end{split}
\end{equation*}
Finally, using \eqref{greatest_integer_simplify}, observe that $\dim(H^{sfk}_{\mathfrak{I}^2_{m,n}})=2b_{\mathfrak{I}^2_{m,n}}-2.$

{$ \boldsymbol{\mathfrak{I}^3_m:}$} Recall that $\Gamma'\subset {\rm U}(2)$ is generated by $[1,\hat{i}]$, $[1,\hat{j}]$, $[e^{\pi i/(3m)}, (-1-\hat{i}-\hat{j}+\hat{k})/2]$.  In Section \ref{U(2)_subgroups},  we saw that these actually generate the entire group $\mathfrak{I}^3_m$, but in order to obtain a smooth quotient it was necessary to take a quotient of $(\mathcal{O}(-2m),g_{LB})$, and included $[e^{\pi i/m},1]$ in the set of generators accordingly.  Therefore, $\Gamma'=\mathfrak{I}^3_m$ is just $\phi(D^*_8\times \langle[e^{\pi i/(3m)},(-1-\hat{i}-\hat{j}+\hat{k})/2]\rangle)$.  To find $\dim(H^1_{\mathfrak{I}^2_{m,n}})$, we first decompose $\Gamma'=\mathfrak{I}^3_m$ into twenty-four disjoint sets of elements
which we write compactly below as $S_{\alpha,s}$ where the index $\alpha \in \{\pm 1,\pm \hat{i},\pm \hat{j}, \pm \hat{k}\}$  and the index $s\in\{0,1,2\}$.
\begin{itemize}
\item $S_{\pm1,s}=\Big\{ \gamma_{\pm1,s}^r:=\pm[e^{\pi i/(3m)},(-1-\hat{i}-\hat{j}+\hat{k})/2]^{3r+s}\Big\}_{0\leq r<m}$\\

\item $S_{\pm\hat{i},s}=\Big\{ \gamma_{\pm\hat{i},s}^r:=[1,\pm \hat{i}]*[e^{\pi i/(3m)},(-1-\hat{i}-\hat{j}+\hat{k})/2]^{3r+s}\Big\}_{0\leq r<m}$\\

\item $S_{\pm\hat{j},s}=\Big\{ \gamma_{\pm\hat{j},s}^r:=[1,\pm \hat{j}]*[e^{\pi i/(3m)},(-1-\hat{i}-\hat{j}+\hat{k})/2]^{3r+s}\Big\}_{0\leq r<m}$\\

\item $S_{\pm\hat{k},s}=\Big\{ \gamma_{\pm\hat{k},s}^r:=[1,\pm \hat{k}]*[e^{\pi i/(3m)},(-1-\hat{i}-\hat{j}+\hat{k})/2]^{3r+s}\Big\}_{0\leq r<m}$\\
\end{itemize}
Any pair of sets with differing index $s$ are clearly disjoint, so to see that these sets are all pairwise disjoint it is enough to observe that
$\gamma_{\pm1,1}^0\neq \gamma_{\pm\hat{i},1}^0\neq \gamma_{\pm\hat{j},1}^0\neq \gamma_{\pm\hat{k},1}^0.$
Since $|S_{\alpha,s}|=m$, for all $\alpha$ and $s$, and $|\mathfrak{I}^3_m|=24m$, we see that 
the disjoint union of these sets constitutes all of $\mathfrak{I}^3_m$.
As $r$ ranges from $0$ to $m-1$, the eigenvalues of the elements in each of these sets can be found to be as follows.
\begin{itemize}
\item
Eigenvalues for $S_{\pm1,s}$: $\pm e^{\frac{\pi i(3r+s)}{3m}}\begin{cases} \{1,1\}\phantom{=}&s=0\\
\{e^{\frac{2\pi i}{3}},e^{\frac{-2\pi i}{3}}\}\phantom{=}&s=1,2.
\end{cases}$

\item Eigenvalues for $S_{\pm \hat{i},s}$: $\pm e^{\frac{\pi i(3r+s)}{3m}}\begin{cases} \{i,-i\}\phantom{=}&s=0\\
\{e^{\frac{\pi i}{3}},e^{\frac{-\pi i}{3}}\}\phantom{=}&s=1\\
\{e^{\frac{2\pi i}{3}},e^{\frac{-2\pi i}{3}}\}\phantom{=}&s=2.
\end{cases}$

\item Eigenvalues for $S_{\pm \hat{j},s}$: $\pm e^{\frac{\pi i(3r+s)}{3m}}\begin{cases} \{i,-i\}\phantom{=}&s=0\\
\{e^{\frac{\pi i}{3}},e^{\frac{-\pi i}{3}}\}\phantom{=}&s=1\\
\{e^{\frac{2\pi i}{3}},e^{\frac{-2\pi i}{3}}\}\phantom{=}&s=2.
\end{cases}$

\item Eigenvalues for $S_{\pm\hat{k},s}$: $\pm e^{\frac{\pi i(3r+s)}{3m}}\begin{cases}\{i,-i\}\phantom{=}&s=0\\
\{e^{\frac{2\pi i}{3}},e^{\frac{-2\pi i}{3}}\}\phantom{=}&s=1,2.
\end{cases}$

\end{itemize}
Using \eqref{index_identity}, the characters are found to be as follows.
\begin{align*}
\bullet\phantom{i} \chi_{\rho}(\gamma^r_{\pm 1,0})&=2m-1\\
\bullet \phantom{i}\chi_{\rho}(\gamma^r_{\pm \hat{i},0})&=\chi_{\rho}(\gamma^r_{\pm \hat{j},0})=\chi_{\rho}(\gamma^r_{\pm \hat{k},0})=1\\
\bullet\phantom{i} \chi_{\rho}(\gamma^r_{\pm 1,1})&=\chi_{\rho}(\gamma^r_{\pm \hat{k},1})=e^{2\pi i/3}\Big[\sin\Big(2(m-1)\frac{2\pi}{3}\Big)\cot\Big(\frac{2\pi}{3}\Big)+\cos\Big(2(m-1)\frac{2\pi}{3}\Big)\Big]\\
&=-e^{2\pi i/3}\phantom{=}\text{since $m\equiv 0$ mod $3$}\\
\bullet\phantom{i} \chi_{\rho}(\gamma^r_{\pm \hat{i},1})&=\chi_{\rho}(\gamma^r_{\pm \hat{j},1})=e^{2\pi i/3}\Big[\sin\Big(2(m-1)\frac{\pi}{3}\Big)\cot\Big(\frac{\pi}{3}\Big)+\cos\Big(2(m-1)\frac{\pi}{3}\Big)\Big]\\
&=-e^{2\pi i/3}\phantom{=}\text{since $m\equiv 0$ mod $3$}\\
\bullet\phantom{i} \chi_{\rho}(\gamma^r_{\pm 1,1})&=\chi_{\rho}(\gamma^r_{\pm \hat{k},1})=\chi_{\rho}(\gamma^r_{\pm \hat{i},1})=\chi_{\rho}(\gamma^r_{\pm \hat{j},1})\\
&=e^{4\pi i/3}\Big[\sin\Big(2(m-1)\frac{2\pi}{3}\Big)\cot\Big(\frac{2\pi}{3}\Big)+\cos\Big(2(m-1)\frac{2\pi}{3}\Big)\Big]\\
&=-e^{4\pi i/3}\phantom{=}\text{since $m\equiv 0$ mod $3$}.
\end{align*}
Finally, we find that
\begin{align*}
\dim(H^{sfk}_{\mathfrak{I}^3_m})&=\frac{1}{12m}\Big[2m\cdot(2m-1)+6m-8m\Big(e^{2\pi i/3}+e^{4\pi i/3}\Big)\Big]\\
&=\frac{m}{3}+1=2b_{\mathfrak{I}^3_m}-2.
\end{align*}
\end{proof}

\subsection{Deformations of complex structure}
\label{DCS}
In this subsection, we give the proof of Theorem~\ref{t2}, and discuss the 
formula \eqref{dimform}.
\begin{proof}[Proof of Theorem \ref{t2}]
From Proposition \ref{sfk_deformation_proposition}, 
for any finite non-cyclic subgroup $\Gamma \subset {\rm{U}}(2)$ containing no 
complex reflections, we see that $(\mathcal{O}(-2m), g_{LB})$ 
always has $H^{sfk}_{\Gamma'}$ of positive dimension. As in \cite{HondaOn, Honda_2014}, 
using equivariant deformation theory of the twistor space of LeBrun 
metrics, the LeBrun metrics consequently have small deformations of complex 
structure admitting scalar-flat K\"ahler ALE metrics, which are 
invariant under the group $\Gamma'$. Choose any such deformed metric, 
and call it~$\tilde{g}$. The quotient of 
$(\mathcal{O}(-2m), \tilde{g})$ by $\Gamma'$ will then be a complex ALE 
orbifold with the same types of singular points as before. 
Consequently, the entire construction in
the proof of Theorem~\ref{t1} can be carried out, using one of these deformations 
as the base metric, to construct scalar-flat K\"ahler ALE metrics
with respect to a small deformation of complex structure
(it is important to note that the proof of Theorem~\ref{SFK_gluing_theorem}
did not assume the existence of holomorphic 
coordinates at infinity on $X_2$). 
Note, the resulting space is diffeomorphic to the minimal 
resolution, but is no longer a minimal resolution,  
since the central rational curve is no longer holomorphic. 

Next, notice that each of the corresponding deformations of complex structure has a 
K\"ahler cone of real dimension $k_{\Gamma}$. If one of these K\"ahler classes
admits a scalar-flat K\"ahler ALE metric near the metric constructed in 
Theorem~\ref{t1}, then it is locally unique; this follows from 
invertibility of the linearized operator shown in the proof 
of Theorem~\ref{t1}. We also use the following: we know there is at least 
one scalar-flat K\"ahler ALE metric with respect to the deformed complex structure, 
so invertibility of the linearized operator as shown in the proof 
of Theorem~\ref{t1} implies that there is an open set in the
K\"ahler cone of K\"ahler classes admitting scalar-flat 
K\"ahler ALE metrics. This is completely analogous to the 
situation considered in \cite{LeBrun_Simanca}; the details are
similar and are omitted.  

The formula \eqref{mdimform} follows since the deformations arisings from $H^{sfk}_{\Gamma'}$ are transverse to the K\"ahler cone, which holds because Honda's deformations of the LeBrun negative mass metrics on $\mathcal{O}(-2m)$ are traceless and do not include scalings of the metrics (the K\"ahler cone is $1$-dimensional in this case) \cite{HondaOn}.
In other words, the parameter space for scalar-flat K\"ahler deformations
of a LeBrun metric on $\mathcal{O}(-n)$ (which is $(n-1)$-dimensional) 
is naturally isomorphic to the parameter space of the Kuranishi family of complex
structures of the Hirzebruch surface $\mathbb{F}_n$, and therefore they are not relevant 
to deformations of a scalar-flat K\"ahler metric caused by moving a K\"ahler class.

\end{proof}

We next briefly discuss the dimension formula \eqref{dimform}. The following argument 
is known, but we give an outline of the proof as it is not easily found by a non-algebraic geometer.
Given a minimal resolution $\tilde{X}$ of $\CC^2 / \Gamma$ 
with exceptional divisor $E= \cup_{i = 1}^k E_i$, with rational curves $E_i$, 
we have the exact sequence 
\begin{align}
0 \rightarrow Der_E ( \tilde{X}) \rightarrow \Theta_{\tilde{X}} 
\rightarrow \oplus_{i = 1}^k \mathcal{O}_{E_i} (E_i) \rightarrow 0,
\end{align}
where $Der_E ( \tilde{X})$ is sheaf of vector fields on $\tilde{X}$ 
dual to $\Omega^1_{\tilde{X}}(\log(E))$, see \cite{Kawamata}. 
We next examine the following portion of 
the associated long exact sequence in cohomology 
\begin{align}
H^1( \tilde{X}, Der_E ( \tilde{X}))
\rightarrow  H^1( \tilde{X},  \Theta_{\tilde{X}}) 
\rightarrow  \oplus_{i = 1}^k H^1 (  \mathcal{O}_{E_i} (E_i)) 
\rightarrow H^2 ( Der_E ( \tilde{X})).
\end{align}
By Serre duality, 
\begin{align}
\begin{split}
\dim( H^1 (  \mathcal{O}_{E_i} (E_i)))  &= \dim( H^0 (\mathbb{P}^1, K \otimes N_{E_i}^*) ) \\
&= \dim (H^0 (\mathcal{O}(e_i -2)) = e_i -1,
\end{split}
\end{align} 
where $E_i \cdot E_i = - e_i$. 
Since $\dim_{\CC}(\tilde{X}) =2$, by Siu's vanishing theorem \cite{Siu}, 
\begin{align}
H^2 ( Der_E ( \tilde{X})) = 0.
\end{align}
Finally, using results from \cite{Brieskorn, Laufer1973, Wahl1975} (see also 
\cite{Behnke_Knorrer}), there
are no nontrivial deformations of $\tilde{X}$ preserving the entire tree
of rational curves, so we have
\begin{align}
H^1( \tilde{X}, Der_E ( \tilde{X})) =0.
\end{align}
Combining the above, this implies that 
\begin{align}
\dim (H^1( \tilde{X},  \Theta_{\tilde{X}}))  = \sum_{i = 1}^k (e_i - 1).
\end{align}

%%%%%%%%%%%%%%%%%%%%%%%%%%%%%%%%%%%%%%%%%%%%%%%%%%%%%%%%
\section{Extremal K\"ahler metrics on rational surfaces}
\label{ca_compactifications}
%%%%%%%%%%%%%%%%%%%%%%%%%%%%%%%%%%%%%%%%%%%%%%%%%%%%%

We next discuss complex analytic compactifications of the Brieskorn 
minimal resolutions. For this, recall the canonical Bochner-K\"ahler metrics on weighted projective spaces, as well as the notion of conformal compactification discussed in Section~\ref{BK_WPS}.  Also, recall that $(\CP^2_{(1,1,n)},g_{BK})$ and $(\mathcal{O}(-n),g_{LB})$ are related via K\"ahler conformal compactification/conformal blow-up \cite{Joyce1991,DabkowskiLock}.

\subsection{Proof of Proposition \ref{cac1intro}}
We consider the weighted projective space $\CP^2_{(1,1,2m)}$ and 
have the following analogue of Theorem~\ref{quotient_theorem}.  
The group $\U(2)$ acts as holomorphic automorphisms. 
For each non-cyclic finite subgroup $\Gamma\subset{\rm U}(2)$ containing no 
complex reflections, the quotient $(\CP^2_{(1,1,2m)},g_{BK})/\Gamma'$ has 
four singularities: one at the point $[0,0,1]$ 
with orbifold group $\Gamma$, and three on the $2$-sphere $\Sigma = [z_0,z_1,0]$ (which, once resolved, 
has self-intersection number $b_{\Gamma}' > 0$ with respect to 
the complex orientation on the weighted projective space) with cyclic orbifold groups orientation reversed conjugate to those given in Theorem \ref{quotient_theorem} for the corresponding $\Gamma$. (Recall that the canonical Bochner-K\"ahler metrics and LeBrun negative mass metrics are K\"ahler with respect to orientation reversed complex structures.)  
Consequently, these orbifold singularities are
of type $L(\beta_i - \alpha_i,\beta_i)$ for $i=1,2,3$.
It is not difficult to see that $\{ \CP^2_{(1,1,2m)} \setminus \Sigma\} / \Gamma'$ is 
biholomorphic to $\CC^2 / \Gamma$.
Let $b^i_j$ be the Hirzebruch-Jung string associated 
to $L(\beta_i - \alpha_i, \beta_i)$ for $i = 1, 2, 3$, and $j = 1 \dots \ell_i$.
Then the compactification is found by adding the 
tree of rational curves at infinity in Figure~\ref{cmpfig}. 
\begin{figure}[h]
\setlength{\unitlength}{2cm}
\begin{picture}(7,3.1)(.5,0)
\linethickness{.3mm}
\put(1,0){\line(1,0){1}}
\put(1,0){\circle*{.1}}
\put(.75,-.25){$-b_{\ell_1}^1$}
\put(2,0){\circle*{.1}}
\put(1.75, -.25){$-b_{\ell_1-1}^1$}
\multiput(2,0)(0.1,0){10}
{\line(1,0){0.05}}
\put(3,0){\circle*{.1}}
\put(2.75, -.25){$-b_{\ell}^1$}
\put(3,0){\line(1,0){1}}
\put(4,0){\circle*{.1}}
\put(3.75, -.25){$+b_{\Gamma}'$}
\put(4.75, -.25){$-b_1^2$}
\put(4,0){\line(1,0){1}}
\put(5,0){\circle*{.1}}
\put(5,0){\line(1,0){1}}
\put(6,0){\circle*{.1}}
\put(5.75,-.25){$-b_2^2$}
\put(7,0){\circle*{.1}}
\put(6.75, -.25){$-b_{\ell_2}^2$}
\multiput(6,0)(0.1,0){10}
{\line(1,0){0.05}}
\put(4,0){\line(0,1){1}}
\put(4,1){\circle*{.1}}
\put(3.55, .75){$-b^3_1$}
\put(4,1){\line(0,1){1}}
\put(4,2){\circle*{.1}}
\put(3.55, 1.75){$-b^3_2$}
\multiput(4,2)(0,0.1){10}
{\line(0,1){0.05}}
\put(4,3){\circle*{.1}}
\put(3.55, 2.75){$-b^3_{\ell_3}$}
\end{picture}
\vspace{5mm}
\caption{The compactification divisor.}\label{cmpfig}
\end{figure}
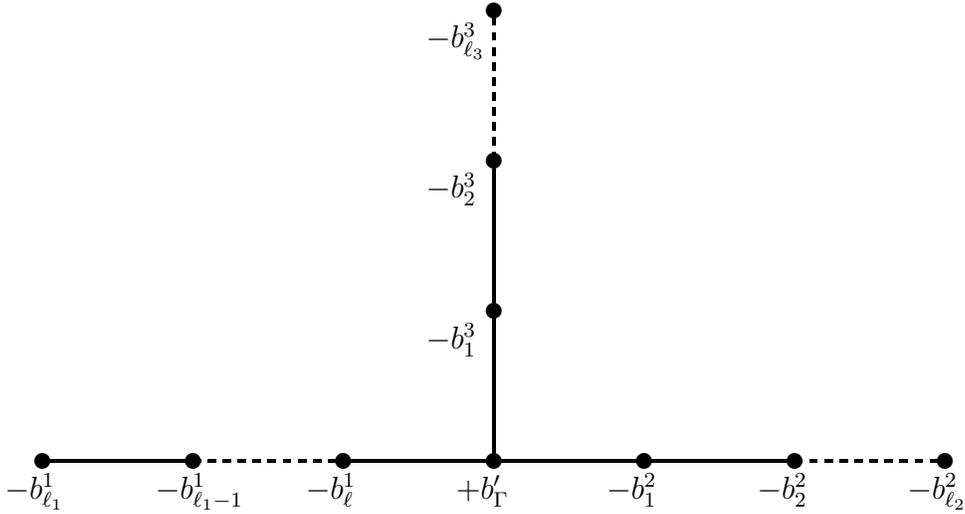

Note that by \cite[Lemma 4.1]{LeBrunMaskit}, a
scalar-flat K\"ahler ALE metric always has a complex analytic compactification 
which is a rational surface. This also follows directly from the above construction since
the compactification is birational to the quotient $\CP^2_{(1,1,2m)}/\Gamma'$, and any finite quotient of a rational surface is rational.
In particular it is diffeomorphic 
to either $S^2 \times S^2$ or $\CP^2 \# \mathfrak{k} \overline{\CP}^2$ for some 
$\mathfrak{k} \geq 0$, and this finishes the proof. 

\subsection{Proof of Theorem \ref{cacintro}}
We begin by noting that the complex analytic compactification of $\mathcal{O}(-n)$ is a 
Hirzebruch surface $\mathbb{F}_n$, obtained by adding the $[+n]$ curve at infinity.
With a minor modification of the above process, we start the process  
with one of Calabi's extremal K\"ahler metrics $g_C$ on the 
Hirzebruch surface $\mathbb{F}_{2m}$ \cite{Calabi1}. This has an isometric 
${\rm{U}}(2)$ action. Arguing as in Theorem \ref{quotient_theorem}, it is not
hard to see that the quotient by the subgroup $\Gamma$ 
produces $6$ singular points. There are $3$ on the quotient of the $[-2m]$ curve 
of type $L(\alpha_i,\beta_i)$ and $3$ on the quotient of the $[2m]$ curve of type 
$L(\beta_i - \alpha_i, \beta_i)$, for $i = 1,2,3$. Note that there is a 
moduli of Calabi's metrics, and the area of the $[2m]$ and $[-2m]$ curves 
can be chosen independently to be any two positive real numbers. 

The existence result then follows from a recent gluing theorem for extremal K\"ahler metrics
due to Arezzo-Lena-Mazzieri \cite{ALM_resolution}.  It is interesting to remark, however, that we do not actually need the full generality of this theorem.  In our case, using the results of \cite{MCC}, we simply note that
the quotient of  $(\mathbb{F}_{2m}, g_C)$ by $\Gamma$ 
has identity component of the isometry group exactly $S^1$, so the Lie algebra 
of Hamiltonian biholomorphisms is $1$-dimensional.  It is therefore trivial to overcome the obstructions, and it follows, as in Arrezo-Pacard-Singer \cite{APSinger}, that we may obtain extremal K\"ahler metrics on the resolution
by attaching Calderbank-Singer metrics at each singular point.
Furthermore, the metrics obtained will admit a holomorphic isometric
$S^1$-action. We note that the Calderbank-Singer metrics come in moduli, 
and the area of each rational curve can be chosen arbitrarily. 
If any of these spaces happens to be Ricci-flat, then the only 
change is that the scaling of this space must be chosen 
differently, which accounts for the factor of $\epsilon^4$ in 
the definition of $f$ and $f'$ above. 

 Finally, since any Calabi metric $g_C$ on $\mathbb{F}_{2m}$ for $m \geq 1$ 
does not have constant scalar curvature, and the extremal metrics we obtain  
are small perturbations of the Calabi metric away from the 
singularities, these extemal metrics do not have constant 
scalar curvature either (for sufficiently small gluing parameter $\varepsilon$).

\begin{remark} {\em
In the above proof, instead of starting with the 
Hirzebruch surface, we could have started with a Bochner-K\"ahler
metric on $\CP^2_{(1,1,2m)}$ (see below), taken the quotient by $\Gamma'$, and 
then used one of the scalar-flat K\"ahler metrics from Theorem~\ref{t1}
to resolve one of the resulting singularities. However, doing this would make 
the $[-b]$ curve have small area. So starting with the 
Calabi metric allows us to obtain existence for a larger
collection of K\"ahler classes. This is of course related
to the fact that the Calabi metrics on $\mathbb{F}_n$
limit to the Bochner-K\"ahler metric on $\CP^2_{(1,1,n)}$ as the 
area of the $[-n]$ curve shrinks to zero \cite{Gauduchon}. 
}
\end{remark}

%%%%%%%%%%%%%%%%%%%%%%%%%%%%%%%%%%%%%%%%%%%%%%%%%%%%%

\section{Hyperk\"ahler metrics}
\label{hyperkahler}

We will now show that
the spaces obtained by our construction for $\Gamma\subset{\rm SU}(2)$ are in fact hyperk\"ahler (and are therefore ALE of order $4$, see \cite{BKN, ct}). 
Note that it follows from \cite{Streets, AV12}
that all of the spaces obtained in the proof of Theorem~\ref{t1}, 
are necessarily ALE of order at least $2$.  Since these spaces are scalar-flat K\"ahler, hence anti-self-dual, we begin with the following Lemma.
\begin{lemma} 
\label{sfasd} 
If $(X,g)$ is a scalar-flat anti-self-dual ALE space with $H^1(X;\mathbb{R})=0$, 
then
\begin{align*}
b_2^-(X)  \geq 2 - \frac{2}{|\Gamma|} - 3 \eta (S^3/\Gamma)
\end{align*}
with equality if and only if $g$ is Ricci-flat and $b_2^+(X)=0$.
\end{lemma}
\begin{proof}
Since $(X,g)$ is an ALE space with $H^1(X;\mathbb{R})=0$, it is clear that
\begin{align}
\tau_{top}(X)=b_2^+-b_2^-\phantom{=}\text{and}\phantom{=}\chi_{top}(X)=1+b_2^++b_2^-.
\end{align}
Now, consider the ALE versions of the signature and Chern-Gauss-Bonnet theorems, which for an ALE space $(M,g)$ with group at infinity $\Gamma\subset {\rm SO}(4)$, are given as follows.
\begin{align}
\label{ALE_signature_euler}
\tau_{top}(M)=\tau_{orb}(M)+\eta(S^3/\Gamma)\phantom{=}\text{and}\phantom{=}\chi_{top}(M)=\chi_{orb}(M)+\frac{1}{|\Gamma|},
\end{align}
where $\tau_{orb}(M)$ is the orbifold signature defined by
\begin{align}
\tau_{orb}(M)=\frac{1}{12\pi^2}\int_M\Big(|W^+_g|^2-|W^-_g|^2\Big)dV_g,
\end{align}
$\eta(S^3/\Gamma)$ is the eta-invariant 
and $\chi_{orb}(M)$ is the orbifold Euler characteristic defined by
\begin{align}
\label{chi_orb}
\chi_{orb}(M)=\frac{1}{8\pi^2}\int_M\Big(|W_g|^2-\frac{|E_g|^2}{2}+\frac{R^2_g}{24}\Big)dV_g.
\end{align}
Then, since $(X,g)$ is anti-self-dual, observe that
\begin{align}
3\tau_{top}(X)+2\chi_{top}(X)=5b_2^+-b_2^-+2=-\frac{1}{8\pi^2}\int_X|E_g|^2dV_g+\frac{2}{|\Gamma|}+3\eta(S^3/\Gamma).
\end{align}
Therefore
\begin{align}
b_2^-=\frac{1}{8\pi^2}\int_X|E_g|^2dV_g-\frac{2}{|\Gamma|}-3\eta(S^3/\Gamma)+2+5b_2^+,
\end{align}
so we find that
\begin{align}
b_2^-(X)  \geq 2 - \frac{2}{|\Gamma|} - 3 \eta (S^3/\Gamma),
\end{align}
with equality if and only if $E_g\equiv 0$, hence $Ric_g\equiv 0$ since $R_g\equiv 0$, and $b_2^+(X)=~0$.
\end{proof}

\begin{proposition}
\label{Kronheimer_sfk_gluing}
For $\Gamma$ a binary polyhedral group, the construction of Theorem \ref{t1} yields a hyperk\"ahler ALE space which admits a holomorphic isometric circle action.
\end{proposition}
\begin{proof}
The proof of Proposition \ref{Kronheimer_sfk_gluing} follows easily from Lemma \ref{sfasd} below as follows. The spaces obtained in Section \ref{t1_proof} are scalar-flat K\"ahler, hence anti-self-dual, and simply connected.  From inserting
the respective $\eta(S^3/\Gamma)$, given in \cite{Nakajima}, and $b_2^-$ into the inequality in Lemma \ref{sfasd}, we find that these spaces are Ricci-flat, and therefore hyperk\"ahler, if and only if $\Gamma\subset {\rm SU}(2)$.
\end{proof}

\begin{remark}
{\em
It is interesting to note that the only ALE spaces needed for in this construction of hyperk\"ahler metrics are the Gibbons-Hawking multi-Eguchi-Hanson metrics.
}
\end{remark}

\begin{remark}{\em
For the weighted projective space $\CP^2_{(r,q,q+r)}$, by a result of Derdzinski \cite{Derdzinski}, the conformal metric $R^{-2}g_{BK}$ is Einstein. A simple calculation shows 
that it is Ricci-flat. This turns out to be the conformal compactification 
of a multi-Eguchi-Hanson metric with 2 points,  one of multiplicity 
$q$ and the other of multiplicity $r$. 
Again, these are K\"ahler with respect to reverse-oriented complex
structures.}
\end{remark}

\subsection{Anti-self-dual gluing}
 We first discuss an ``inverse'' to the process of conformal compactification 
described above.  Given a compact Riemannian orbifold $(\widehat{X},\widehat{g})$ with nonnegative Yamabe invariant, and letting $G_p$ denote the Green's function of the conformal Laplacian associated to a point $p$, the space $(\widehat{X}\setminus\{p\},G_p^2\widehat{g})$ is a complete noncompact scalar-flat anti-self-dual ALE orbifold.  A coordinate system at infinity arises from using inverted normal coordinates in the metric $\widehat{g}$ around $p$.  We refer to this as a {\textit{conformal blow-up}}.

 In Proposition \ref{Kronheimer_sfk_gluing}, we obtained hyperk\"ahler ALE metrics with groups at infinity the binary polyhedral groups (the ``DE'' in ``ADE'').  For these cases, the singularities on the $\CP^1$ at the origin of Theorem \ref{quotient_theorem} are all of type $L(-1,\beta_i)$, and we used the scalar-flat K\"ahler gluing theorem of Section~\ref{SFK_gluing} to glue on the appropriate multi-Eguchi-Hanson metrics (the ``A'' in ``ADE'') to the quotients of $(\mathcal{O}(-2),g_{LB})$.  However, this construction 
only produces {\textit{some}} hyperk\"ahler metrics. 
We will next show that we can obtain examples for {\textit{all}}  small 
deformations of complex structure, and therefore an open set in the moduli space, by using the following gluing result for (anti-)self-dual orbifolds.
\begin{theorem}[\cite{DonaldsonFriedman,Floer,LeBrunSinger,KovalevSinger,AcheViaclovsky2}]
\label{SD_gluing}
Let $(X_1,[g_1])$ and $(X_2,[g_2])$ be self-dual conformal structures on compact $4$-dimensional orbifolds with $H^2_{SD}$ of the respective self-dual deformation complexes vanishing.  Let $p_1\in X_1$ and $p_2\in X_2$ be orbifold points with the respective orbifold groups $\Gamma_1$ and
$\Gamma_2\subset{\rm SO}(4)$ orientation reversed conjugate in the sense that there is an orientation-reversing intertwining map between these groups.  Then, the orbifold connect sum $X_1\# X_2$, taken at these points, admits self-dual conformal structures.
\end{theorem}

We will replace the scalar-flat K\"ahler gluing result in Theorem \ref{SFK_gluing} with 
Theorem~\ref{SD_gluing}.
For any binary polyhedral group $\Gamma'\subset {\rm SU}(2)$, consider the compact 
self-dual K\"ahler orbifold $(\CP^2_{(1,1,2)},g_{BK})/\Gamma'$, as discussed above, which has four 
orbifold points: one with orbifold group $\Gamma$ at the point of compactification, 
and three on the $2$-sphere at infinity with cyclic orbifold groups $L(1,\beta_i)$ for $i=1,2,3$.
Recall that these are just the K\"ahler conformal compactifications of the 
corresponding quotients of the Eguchi-Hanson metric.  
LeBrun-Maskit proved that $H^2_{SD}$ 
of the compactification of any scalar-flat K\"ahler ALE metric vanishes
 \cite[Theorem 4.2]{LeBrunMaskit}.  Therefore, to each orbifold point with 
group $L(1,\beta_i)$ on the $2$-sphere at infinity of $(\CP^2_{(1,1,2m)},g_{BK})/\Gamma'$, 
Theorem~\ref{SD_gluing} can be used to attach the conformal 
compactification of the multi-Eguchi-Hanson metric with group $L(-1,\beta_i)$ at 
infinity, and obtain a self-dual orbifold with a single orbifold point of type 
$\Gamma$, and positive Yamabe invariant.  Since the Yamabe invariant is positive, 
we can take the conformal blow-up at the orbifold point to obtain a scalar-flat 
anti-self-dual ALE metric on the minimal resolution of $\CC^2/\Gamma$.  
Once again, notice that the only building blocks here were the 
multi-Eguchi-Hanson metrics.

Finally, for each binary polyhedral group $\Gamma$, we examine these scalar-flat anti-self-dual ALE metrics using Lemma~\ref{sfasd}.  As in the proof of Proposition \ref{Kronheimer_sfk_gluing}, from inserting $\eta(S^3/\Gamma)$ 
given in \cite{Nakajima} and $b_2^-$ into the inequality of Lemma \ref{sfasd}, it is clear that these metrics are necessarily Ricci-flat, and are therefore hyperk\"ahler ALE metrics.

\bibliography{Smorgasbord_references}

\def\cprime{$'$}
\providecommand{\bysame}{\leavevmode\hbox to3em{\hrulefill}\thinspace}
\providecommand{\MR}{\relax\ifhmode\unskip\space\fi MR }
\providecommand{\MRhref}[2]{%
  \href{http://www.ams.org/mathscinet-getitem?mr=#1}{#2}
}
\providecommand{\href}[2]{#2}
\begin{thebibliography}{ALM15b}

\bibitem[ALM15a]{ALM_kummer}
Claudio Arezzo, Riccardo Lena, and Lorenzo Mazzieri, \emph{On the {K}ummer
  construction for {K}csc metrics}, arXiv.org:1507.05105, 2015.

\bibitem[ALM15b]{ALM_resolution}
\bysame, \emph{On the resolution of extremal and constant scalar curvature
  {K}\"ahler orbifolds}, Int. Math. Res. Not. IMRN (2015), no.~RNV346, 1--38.

\bibitem[AP06]{API}
Claudio Arezzo and Frank Pacard, \emph{Blowing up and desingularizing constant
  scalar curvature {K}\"ahler manifolds}, Acta Math. \textbf{196} (2006),
  no.~2, 179--228.

\bibitem[AP09]{APII}
\bysame, \emph{Blowing up {K}\"ahler manifolds with constant scalar curvature.
  {II}}, Ann. of Math. (2) \textbf{170} (2009), no.~2, 685--738.

\bibitem[Apo90]{Apostol}
Tom~M. Apostol, \emph{Modular functions and {D}irichlet series in number
  theory}, second ed., Graduate Texts in Mathematics, vol.~41, Springer-Verlag,
  New York, 1990.

\bibitem[APS11]{APSinger}
Claudio Arezzo, Frank Pacard, and Michael Singer, \emph{Extremal metrics on
  blowups}, Duke Math. J. \textbf{157} (2011), no.~1, 1--51.

\bibitem[AR15]{Apostolov-Rollin}
Vestislav Apostolov and Yann Rollin, \emph{{ALE} scalar-flat {K}\"ahler metrics
  on non-compact weighted projective spaces}, arXiv.org:1510.02226, to appear
  in Math. Ann., 2015.

\bibitem[AV12]{AV12}
Antonio~G. Ache and Jeff~A. Viaclovsky, \emph{Obstruction-flat asymptotically
  locally {E}uclidean metrics}, Geom. Funct. Anal. \textbf{22} (2012), no.~4,
  832--877.

\bibitem[AV15]{AcheViaclovsky2}
\bysame, \emph{Asymptotics of the self-dual deformation complex}, J. Geom.
  Anal. \textbf{25} (2015), no.~2, 951--1000.

\bibitem[BK87]{Behnke_Knorrer}
Kurt Behnke and Horst Kn{\"o}rrer, \emph{On infinitesimal deformations of
  rational surface singularities}, Compositio Math. \textbf{61} (1987), no.~1,
  103--127.

\bibitem[BKN89]{BKN}
Shigetoshi Bando, Atsushi Kasue, and Hiraku Nakajima, \emph{On a construction
  of coordinates at infinity on manifolds with fast curvature decay and maximal
  volume growth}, Invent. Math. \textbf{97} (1989), no.~2, 313--349.

\bibitem[Boy86]{Boyer}
Charles~P. Boyer, \emph{Conformal duality and compact complex surfaces}, Math.
  Ann. \textbf{274} (1986), no.~3, 517--526.

\bibitem[BR95]{Behnke-Riemenschneider}
Kurt Behnke and Oswald Riemenschneider, \emph{Quotient surface singularities
  and their deformations}, Singularity theory ({T}rieste, 1991), World Sci.
  Publ., River Edge, NJ, 1995, pp.~1--54.

\bibitem[BR15]{BiquardRollin}
Olivier Biquard and Yann Rollin, \emph{Smoothing singular extremal {K}\"ahler
  surfaces and minimal {L}agrangians}, Advances in Math. \textbf{285} (2015),
  980--1024.

\bibitem[Bri68]{Brieskorn}
Egbert Brieskorn, \emph{Rationale {S}ingularit\"aten komplexer {F}l\"achen},
  Invent. Math. \textbf{4} (1967/1968), 336--358.

\bibitem[Bry01]{Bryant}
Robert~L. Bryant, \emph{Bochner-{K}\"ahler metrics}, J. Amer. Math. Soc.
  \textbf{14} (2001), no.~3, 623--715.

\bibitem[Bur86]{Burns}
Daniel Burns, \emph{Twistors and harmonic maps}, Talk in Charlotte, N.C.,
  October 1986.

\bibitem[Cal82]{Calabi1}
Eugenio Calabi, \emph{Extremal {K}\"ahler metrics}, Seminar on {D}ifferential
  {G}eometry, Ann. of Math. Stud., vol. 102, Princeton Univ. Press, Princeton,
  N.J., 1982, pp.~259--290.

\bibitem[Cal85]{Calabi2}
\bysame, \emph{Extremal {K}\"ahler metrics. {II}}, Differential geometry and
  complex analysis, Springer, Berlin, 1985, pp.~95--114.

\bibitem[CH14]{ConlonHeinIII}
Ronan~J. Conlon and Hans-Joachim Hein, \emph{Asymptotically conical
  {C}alabi-{Y}au manifolds, {III}}, arXiv.org:1405.7140, 2014.

\bibitem[CLW08]{ChenLeBrunWeber}
Xiuxiong Chen, Claude Lebrun, and Brian Weber, \emph{On conformally {K}\"ahler,
  {E}instein manifolds}, J. Amer. Math. Soc. \textbf{21} (2008), no.~4,
  1137--1168.

\bibitem[Cox40]{Coxeter_1940}
H.~S.~M. Coxeter, \emph{The binary polyhedral groups, and other generalizations
  of the quaternion group}, Duke Math. J. \textbf{7} (1940), 367--379.

\bibitem[Cox91]{Coxeter}
\bysame, \emph{Regular complex polytopes}, 2nd ed., Cambridge University Press,
  Cambridge, 1991.

\bibitem[Cro61]{Crowe}
D.~W. Crowe, \emph{The groups of regular complex polygons}, Canad. J. Math.
  \textbf{13} (1961), 149--156.

\bibitem[CS04]{CalderbankSinger}
David M.~J. Calderbank and Michael~A. Singer, \emph{Einstein metrics and
  complex singularities}, Invent. Math. \textbf{156} (2004), no.~2, 405--443.

\bibitem[CT94]{ct}
Jeff Cheeger and Gang Tian, \emph{On the cone structure at infinity of {R}icci
  flat manifolds with {E}uclidean volume growth and quadratic curvature decay},
  Invent. Math. \textbf{118} (1994), no.~3, 493--571.

\bibitem[Der83]{Derdzinski}
Andrzej Derdzi{\'n}ski, \emph{Self-dual {K}\"ahler manifolds and {E}instein
  manifolds of dimension four}, Compositio Math. \textbf{49} (1983), no.~3,
  405--433.

\bibitem[DF89]{DonaldsonFriedman}
S.~Donaldson and R.~Friedman, \emph{Connected sums of self-dual manifolds and
  deformations of singular spaces}, Nonlinearity \textbf{2} (1989), no.~2,
  197--239.

\bibitem[DG06]{DavidGauduchon}
Liana David and Paul Gauduchon, \emph{The {B}ochner-flat geometry of weighted
  projective spaces}, Perspectives in {R}iemannian geometry, CRM Proc. Lecture
  Notes, vol.~40, Amer. Math. Soc., Providence, RI, 2006, pp.~109--156.

\bibitem[DL14]{DabkowskiLock}
Michael~G. Dabkowski and Michael~T. Lock, \emph{On {K}\"ahler conformal
  compactifications of ${U}(n)$-invariant {ALE} spaces}, to appear in Ann.
  Global Anal. Geom., 2014.

\bibitem[DV64]{DuVal}
Patrick Du~Val, \emph{Homographies, quaternions and rotations}, Oxford
  Mathematical Monographs, Clarendon Press, Oxford, 1964.

\bibitem[EH79]{EguchiHanson}
Tohru Eguchi and Andrew~J. Hanson, \emph{Self-dual solutions to {E}uclidean
  gravity}, Ann. Physics \textbf{120} (1979), no.~1, 82--106.

\bibitem[Flo91]{Floer}
Andreas Floer, \emph{Self-dual conformal structures on {$l{\bf C}{\rm P}\sp
  2$}}, J. Differential Geom. \textbf{33} (1991), no.~2, 551--573.

\bibitem[FP04]{FalbelPaupert}
Elisha Falbel and Julien Paupert, \emph{Fundamental domains for finite
  subgroups in {$U(2)$} and configurations of {L}agrangians}, Geom. Dedicata
  \textbf{109} (2004), 221--238.

\bibitem[Gau09]{Gauduchon}
Paul Gauduchon, \emph{Hirzebruch surfaces and weighted projective planes},
  Riemannian topology and geometric structures on manifolds, Progr. Math., vol.
  271, Birkh\"auser Boston, Boston, MA, 2009, pp.~25--48.

\bibitem[GH78]{GibbonsHawking}
G.~W. Gibbons and S.~W. Hawking, \emph{Gravitational multi-instantons}, Physics
  Letters B \textbf{78} (1978), no.~4, 430--432.

\bibitem[Hir53]{Hirzebruch1953}
Friedrich Hirzebruch, \emph{\"{U}ber vierdimensionale {R}iemannsche {F}l\"achen
  mehrdeutiger analytischer {F}unktionen von zwei komplexen
  {V}er\"anderlichen}, Math. Ann. \textbf{126} (1953), 1--22.

\bibitem[Hon13]{HondaOn}
Nobuhiro Honda, \emph{Deformation of {L}e{B}run's {ALE} metrics with negative
  mass}, Comm. Math. Phys. \textbf{322} (2013), no.~1, 127--148.

\bibitem[Hon14]{Honda_2014}
\bysame, \emph{Scalar flat {K}\"ahler metrics on affine bundles over
  {$\mathbb{CP}^1$}}, SIGMA Symmetry Integrability Geom. Methods Appl.
  \textbf{10} (2014), Paper 046, 25.

\bibitem[HP78]{HawkingPope}
S.~W. Hawking and C.~N. Pope, \emph{Symmetry breaking by instantons in
  supergravity}, Nuclear Phys. B \textbf{146} (1978), no.~2, 381--392.

\bibitem[Joy91]{Joyce1991}
Dominic Joyce, \emph{The hypercomplex quotient and the quaternionic quotient},
  Math. Ann. \textbf{290} (1991), no.~2, 323--340.

\bibitem[Joy95]{Joyce1995}
\bysame, \emph{Explicit construction of self-dual {$4$}-manifolds}, Duke Math.
  J. \textbf{77} (1995), no.~3, 519--552.

\bibitem[Kaw78]{Kawamata}
Yujiro Kawamata, \emph{On deformations of compactifiable complex manifolds},
  Math. Ann. \textbf{235} (1978), no.~3, 247--265.

\bibitem[Kro89a]{Kronheimer}
P.~B. Kronheimer, \emph{The construction of {A}{L}{E} spaces as
  hyper-{K}\"ahler quotients}, J. Differential Geom. \textbf{29} (1989), no.~3,
  665--683.

\bibitem[Kro89b]{Kronheimer2}
\bysame, \emph{A {T}orelli-type theorem for gravitational instantons}, J.
  Differential Geom. \textbf{29} (1989), no.~3, 685--697.

\bibitem[KS01]{KovalevSinger}
A.~Kovalev and M.~Singer, \emph{Gluing theorems for complete anti-self-dual
  spaces}, Geom. Funct. Anal. \textbf{11} (2001), no.~6, 1229--1281.

\bibitem[Lam03]{Lam}
T.~Y. Lam, \emph{Hamilton's quaternions}, Handbook of algebra, {V}ol. 3, Handb.
  Algebr., vol.~3, Elsevier/North-Holland, Amsterdam, 2003, pp.~429--454.

\bibitem[Lau73]{Laufer1973}
Henry~B. Laufer, \emph{Taut two-dimensional singularities}, Math. Ann.
  \textbf{205} (1973), 131--164.

\bibitem[LeB88]{LeBrunnegative}
Claude LeBrun, \emph{Counter-examples to the generalized positive action
  conjecture}, Comm. Math. Phys. \textbf{118} (1988), no.~4, 591--596.

\bibitem[Li14]{ChiLi}
Chi Li, \emph{On sharp rates and analytic compactifications of asymptotically
  conical {K}\"ahler metrics}, arXiv.org:1405.2433, 2014.

\bibitem[LM08]{LeBrunMaskit}
Claude LeBrun and Bernard Maskit, \emph{On optimal 4-dimensional metrics}, J.
  Geom. Anal. \textbf{18} (2008), no.~2, 537--564.

\bibitem[LS94a]{LeBrun_Simanca}
C.~LeBrun and S.~R. Simanca, \emph{Extremal {K}\"ahler metrics and complex
  deformation theory}, Geom. Funct. Anal. \textbf{4} (1994), no.~3, 298--336.

\bibitem[LS94b]{LeBrunSinger}
Claude LeBrun and Michael Singer, \emph{A {K}ummer-type construction of
  self-dual {$4$}-manifolds}, Math. Ann. \textbf{300} (1994), no.~1, 165--180.

\bibitem[McC02]{MCC}
Darryl McCullough, \emph{Isometries of elliptic 3-manifolds}, J. London Math.
  Soc. (2) \textbf{65} (2002), no.~1, 167--182.

\bibitem[McK80]{McKay}
John McKay, \emph{Graphs, singularities, and finite groups}, The {S}anta {C}ruz
  {C}onference on {F}inite {G}roups ({U}niv. {C}alifornia, {S}anta {C}ruz,
  {C}alif., 1979), Proc. Sympos. Pure Math., vol.~37, Amer. Math. Soc.,
  Providence, R.I., 1980, pp.~183--186.

\bibitem[Nak90]{Nakajima}
Hiraku Nakajima, \emph{Self-duality of {A}{L}{E} {R}icci-flat $4$-manifolds and
  positive mass theorem}, Recent topics in differential and analytic geometry,
  Academic Press, Boston, MA, 1990, pp.~385--396.

\bibitem[OR70]{OrlikRaymond}
Peter Orlik and Frank Raymond, \emph{Actions of the torus on {$4$}-manifolds.
  {I}}, Trans. Amer. Math. Soc. \textbf{152} (1970), 531--559.

\bibitem[RS05]{RollinSinger}
Yann Rollin and Michael Singer, \emph{Non-minimal scalar-flat {K}\"ahler
  surfaces and parabolic stability}, Invent. Math. \textbf{162} (2005), no.~2,
  235--270.

\bibitem[R{\c{S}}15]{Suvaina}
R.~R{\u{a}}sdeaconu and I.~{\c{S}}uvaina, \emph{A{LE} {R}icci-flat {K}\"ahler
  surfaces and weighted projective spaces}, Ann. Global Anal. Geom. \textbf{47}
  (2015), no.~2, 117--134.

\bibitem[Sco83]{Scott}
Peter Scott, \emph{The geometries of {$3$}-manifolds}, Bull. London Math. Soc.
  \textbf{15} (1983), no.~5, 401--487.

\bibitem[Sei33]{Seifert}
H.~Seifert, \emph{Topologie {D}reidimensionaler {G}efaserter {R}\"aume}, Acta
  Math. \textbf{60} (1933), no.~1, 147--238.

\bibitem[Siu69]{Siu}
Yum-tong Siu, \emph{Analytic sheaf cohomology groups of dimension {$n$} of
  {$n$}-dimensional noncompact complex manifolds}, Pacific J. Math. \textbf{28}
  (1969), 407--411.

\bibitem[Ste08]{Stekolshchik}
R.~Stekolshchik, \emph{Notes on {C}oxeter transformations and the {M}c{K}ay
  correspondence}, Springer Monographs in Mathematics, Springer-Verlag, Berlin,
  2008.

\bibitem[Str10]{Streets}
Jeffrey Streets, \emph{Asymptotic curvature decay and removal of singularities
  of {B}ach-flat metrics}, Trans. Amer. Math. Soc. \textbf{362} (2010), no.~3,
  1301--1324.

\bibitem[Sz{\'e}12]{GaborI}
G{\'a}bor Sz{\'e}kelyhidi, \emph{On blowing up extremal {K}\"ahler manifolds},
  Duke Math. J. \textbf{161} (2012), no.~8, 1411--1453.

\bibitem[Sz{\'e}15]{GaborII}
\bysame, \emph{Blowing up extremal {K}\"ahler manifolds {II}}, Invent. Math.
  \textbf{200} (2015), no.~3, 925--977.

\bibitem[TS31]{Threlfall-Seifert_1}
W.~Threlfall and H.~Seifert, \emph{Topologische {U}ntersuchung der
  {D}iskontinuit\"atsbereiche endlicher {B}ewegungsgruppen des
  dreidimensionalen sph\"arischen {R}aumes}, Math. Ann. \textbf{104} (1931),
  no.~1, 1--70.

\bibitem[TS33]{Threlfall-Seifert_2}
\bysame, \emph{Topologische {U}ntersuchung der {D}iskontinuit\"atsbereiche
  endlicher {B}ewegungsgruppen des dreidimensionalen sph\"arischen {R}aumes
  ({S}chlu\ss)}, Math. Ann. \textbf{107} (1933), no.~1, 543--586.

\bibitem[TV05]{TV2}
Gang Tian and Jeff Viaclovsky, \emph{Moduli spaces of critical {R}iemannian
  metrics in dimension four}, Adv. Math. \textbf{196} (2005), no.~2, 346--372.

\bibitem[Via10]{ViaclovskyFourier}
Jeff Viaclovsky, \emph{Monopole metrics and the orbifold {Y}amabe problem},
  Annales de L'Institut Fourier \textbf{60} (2010), no.~7, 2503--2543.

\bibitem[Wah75]{Wahl1975}
Jonathan~M. Wahl, \emph{Vanishing theorems for resolutions of surface
  singularities}, Invent. Math. \textbf{31} (1975), no.~1, 17--41.

\bibitem[Wri11]{Wright}
Dominic Wright, \emph{Compact anti-self-dual orbifolds with torus actions},
  Selecta Math. (N.S.) \textbf{17} (2011), no.~2, 223--280.

\end{thebibliography}
\end{document}